\documentclass{article}
\usepackage{amsfonts}
\usepackage{amssymb}
\usepackage{amsmath}
\usepackage{makeidx}

\newtheorem{theorem}{Theorem}

\newtheorem{condition}[theorem]{Condition}

\newtheorem{corollary}[theorem]{Corollary}

\newtheorem{definition}[theorem]{Definition}

\newtheorem{lemma}[theorem]{Lemma}

\newtheorem{proposition}[theorem]{Proposition}
\newtheorem{remark}[theorem]{Remark}

\newenvironment{proof}[1][Proof]{\noindent\textbf{#1.} }{\ \rule{0.5em}{0.5em}}
\input{tcilatex}
\begin{document}

\title{The filtering equations revisited}
\author{Thomas Cass\thanks{%
Department of Mathematics, Imperial College London.} \and Martin Clark%
\thanks{%
Department of Electrical and Electronic Engineering, Imperial College London.%
} \and Dan Crisan\thanks{%
Department of Mathematics, Imperial College London.}}
\date{}
\maketitle

\begin{abstract}
The problem of nonlinear filtering has engendered a surprising number of
mathematical techniques for its treatment. A notable example is the
change-of--probability-measure method introduced by Kallianpur
and Striebel to derive the filtering equations and the Bayes-like formula
that bears their names. More recent work, however, has generally preferred
other methods. In this paper, we reconsider the change-of-measure approach
to the derivation of the filtering equations and show that many of the
technical conditions present in previous work can be relaxed. The filtering
equations are established for general Markov signal processes that can be
described by a martingale-problem formulation. Two specific applications are treated.

\end{abstract}

\section{Introduction}

The aim of nonlinear filtering is to estimate an evolving dynamical system,
customarily modelled by a stochastic process and called the signal process.
The signal process cannot be measured directly, but only via a related
process, termed the observation process. The filtering problem consists in
computing the conditional distribution of the signal at the current time
given the observation data accumulated up to that time. In order to describe
the contribution of the paper, we start with a few historical comments on
the subject.

The development of the modern theory of nonlinear filtering started in the
sixties with the publications of Stratonovich \cite{rls1,rls2}, Kushner \cite%
{ku1,ku2} and Shiryaev \cite{shi} for diffusions and Wonham for pure-jump
Markov processes \cite{w}; these introduced the basic form of the class of
stochastic differential equations for the conditional distributions of
partially observed Markov processes, which are now known generically as the
filtering equation. This class of equations has inspired authors to
introduce a rich variety of mathematical techniques to justify their
structure, together with that of their un-normalized form, the Zakai (or
Duncan-Mortensen-Zakai) equation, \cite{d,m,z}, and to establish the
existence, uniqueness and regularity of their solutions. A description of
much of the work on this equation and its generalizations can be found in 
\cite{kal} for papers before 1980, in \cite{ls2,ls1} for papers before 2000
and in \cite{bc,cr,x} for more recent work.

For instance, Fujisaki, Kallianpur and Kunita \cite{fkk} \ exploited a
stochastic-integral representation theorem in order to enable them to
express conditional distributions as functionals of an \textquotedblleft
innovations\textquotedblright\ martingale (a concept introduced in the
Gaussian case by Kailath \cite{kai}). Krylov, Rozovsky, Pardoux \cite%
{kr1,kr2,pp}, Chapter 6 in \cite{cr} and other authors developed a general
theory of stochastic partial differential equations that led to a direct
`PDE' approach to the filtering equations, but there are many other
approaches For example, see the work of Grigelionis and Mikulevicius on
filtering for signal and observation processes with jumps [4, Chapter 4] and
that of Kurtz and Nappo on the filtered martingale problem [4 Chapter 5].

In parallel with the above developments, Snyder \cite{snyder}, Br\'{e}maud 
\cite{bremaud} and van Schuppen \cite{sthesis} have initiated the study of
the filtering problem for observations of counting process type. A large
number of papers have been written on this class of filtering problems. Some
of the early contributors to this topic include Boel, Davis, Segal, Varaiya,
Willems and Wong, see \cite{bvw, davis, segall,vS, WvS,vSW}. Also,
Grigelionis \cite{grig} looked at filtering problems with common jumps of
the unobserved state process and of the observations. For further
developments in this directions see [4 Chapter 10].

A probabilistic approach, initially considered formally by Bucy \cite{bucy},
but developed in detail by Kallianpur and Striebel \cite{ks1,ks2}, made use
of a functional form of Bayes formula for processes, now known as the
Kallianpur-Striebel formula. This technique, which is based on a change of
probability measure that makes, at each time, the future observation process
independent of past processes, is effective for filtering problems in which
the observation process is of the \textquotedblleft signal plus white
noise\textquotedblright\ variety, where the signal is independent of the
noise process, but less so for the \textquotedblleft correlated
case\textquotedblright ; that is, for problems in which observed and
unobserved components are coupled via a common noise process. For this
reason, among probabilistic methods, the \textquotedblleft
innovations\textquotedblright\ approach is often preferred to the
\textquotedblleft change of measure\textquotedblright\ method. The
awkwardness in its application results from the fact that an exponential
local martingale, constructed via Girsanov's theorem as a process of
potential densities, has to be verified as a true martingale, and this is
generally requires ad hoc techniques peculiar to the particular filtering
problem being considered.

In this paper we re-visit the change-of-measure method and show that it can
be used to derive the filtering equations for a broad class of Markov
processes with coupled observed and unobserved components. This class
includes diffusions with jumps obeying only mild linear growth conditions on
their characteristic coefficients. Propositions are also presented that
serve to test whether the filtering equations are derivable by the
change-of-measure method for a particular filtering problem.

\section{The Filtering Framework}

Let $(\Omega ,\mathcal{F},\mathbb{P})$ be a probability space together with
a filtration $(\mathcal{F}_{t})_{t\geq 0}$ which satisfies the usual
conditions\footnote{%
The probability space $(\Omega ,\mathcal{F},\mathbb{P})$ together with the
filtration $(\mathcal{F}_{t})_{t\geq 0}$ satisfies the usual conditions
provided: a. $\mathcal{F}$ is complete i.e. $A\subset B$, $B\in \mathcal{F}$
and $\mathbb{P}(B)=0$ implies that $A\in \mathcal{F}$ and $\mathbb{P}(A)=0$,
b. The filtration $\mathcal{F}_{t}$ is right continuous i.e. $\mathcal{F}%
_{t}=\mathcal{F}_{t+}$. c. $\mathcal{F}_{0}$ (and consequently all $\mathcal{%
F}_{t}$ for $t\geq 0$) contains all the $\mathbb{P}$-null sets.}. On $%
(\Omega ,\mathcal{F},\mathbb{P})$ we consider an $\mathcal{F}_{t}$-adapted
process $\bar{X}$ with c\`{a}dl\`{a}g paths. The process $\bar{X}$ consists
in a pair of processes $X$ and $Y$, $\bar{X}=\left( X,Y\right) $. The
process $X$ is called the \emph{signal} process and is assumed to take
values in a complete separable metric space $\mathbb{S}$ (the state space).
The process $Y$ is assumed to take values in $\mathbb{R}^{m}$ and is called
the \emph{observation} process.

Let $\mathcal{B}(\mathbb{S\times R}^{m})$ be the associated product Borel $%
\sigma $-algebra on $\mathbb{S\times R}^{m}$ and $b\mathcal{B}(\mathbb{%
S\times R}^{m})$ be the space of bounded $\mathcal{B}(\mathbb{S\times R}%
^{m}) $-measurable functions. Let $A\colon b\mathcal{B}(\mathbb{S\times R}%
^{m})\rightarrow b\mathcal{B}(\mathbb{S\times R}^{m})$ and write $\mathcal{D}%
(A)\subseteq b\mathcal{B}(\mathbb{S\times R}^{m})$ for the domain of $A$. We
assume that $\mathbf{1}\in \mathcal{D}(A)$ and $A\mathbf{1}=0$. In the
following we will assume that the distribution of $X_{0}$ is $\pi _{0}\in 
\mathcal{P}(\mathbb{S})$ and that $Y_{0}=0$. Since $Y_{0}=0,$ the initial
distribution of $X$, is identical with the conditional distribution of $%
X_{0} $ given $\mathcal{Y}_{0}$ and we use the same notation for both.
Further we will assume that $\bar{X}$ is a solution of the martingale
problem for $(A,\pi _{0}\times \delta _{0})$. In other words, we assume that
the process $M^{\varphi }=\{M_{t}^{\varphi },\ t\geq 0\}$ defined as 
\begin{equation}
M_{t}^{\varphi }=\varphi (\bar{X}_{t})-\varphi (\bar{X}_{0})-\int_{0}^{t}A%
\varphi (\bar{X}_{s}){\mathrm{d}}s,\quad \ t\geq 0,
\label{generatorforXand Y}
\end{equation}%
is an $\mathcal{F}_{t}$-adapted martingale for any $\varphi \in \mathcal{D}%
(A)$. In addition, let $h=(h_{i})_{i=1}^{m}:\mathbb{S}\rightarrow \mathbb{R}%
^{m}$ be a measurable function such that 
\begin{equation}
P\left( \int_{0}^{t}\left\vert h^{i}(\bar{X}_{s})\right\vert ^{2}ds<\infty
\right) =1.  \label{weakh}
\end{equation}%
for all $t\geq 0$. Let $W$ be a standard $\mathcal{F}_{t}$-adapted $m$%
-dimensional Brownian motion defined on $(\Omega ,\mathcal{F},\mathbb{P})$.
We will assume that $Y$ satisfies the following evolution equation 
\begin{equation}
Y_{t}=\int_{0}^{t}h(\bar{X}_{s})\,\mathrm{d}s+W_{t}.
\label{eq:filterEq:observation}
\end{equation}%
To complete the description we need to identify the covariation process
between $M^{\varphi }=\{M_{t}^{\varphi },\ t\geq 0\}$ and $W$. For this we
introduce $m$ operators $B^{i}\colon b\mathcal{B}(\mathbb{S\times R}%
^{m})\rightarrow b\mathcal{B}(\mathbb{S\times R}^{m}),$ $i=1,...,m$ with $%
\mathcal{D}(A)\subseteq \mathcal{D}(B^{i})\subseteq b\mathcal{B}(\mathbb{%
S\times R}^{m})$. We assume that $\mathbf{1}\in \mathcal{D}(A)$ and $A%
\mathbf{1}=0$. We will assume that, 
\begin{equation}
\left\langle M^{\varphi },W^{i}\right\rangle _{t}=\int_{0}^{t}B^{i}\varphi
\left( \bar{X}_{s}\right) ds+\int_{0}^{t}\frac{\partial \varphi }{\partial
y_{i}}\left( \bar{X}_{s}\right) ds,  \label{covariance}
\end{equation}%
for any $t\geq 0$ and for test functions $\varphi $ both in the domain of $A$
and with bounded partial derivatives in the $y$ direction. In particular,
for functions that are constant in the second component, then we have 
\begin{equation}
\left\langle M^{\varphi },W\right\rangle _{t}=\int_{0}^{t}B^{i}\varphi
\left( X_{s},Y_{s}\right) ds.  \label{covariance2}
\end{equation}%
Let $\{\mathcal{Y}_{t},\ t\geq 0\}$ be the usual augmentation of the
filtration associated with the process $Y$, viz 
\begin{equation}
\mathcal{Y}_{t}=\bigcap_{\varepsilon >0}\sigma (Y_{s},\ s\in \lbrack
0,t+\varepsilon ])\vee \mathcal{N},~~~\mathcal{Y}=\bigvee_{t\in \mathbb{R}%
_{+}}\mathcal{Y}_{t}.  \label{yt}
\end{equation}%
where $\mathcal{N}$ is that class of all $\mathbb{P}$-null sets. Note that $%
Y_{t}$ is $\mathcal{F}_{t}$-adapted, hence $\mathcal{Y}_{t}\subset \mathcal{F%
}_{t}$.\ In the following we will assume that $\mathcal{Y}_{t}$ is a right
continuous filtration.

\begin{definition}
The filtering problem consists in determining the conditional distribution $%
\pi _{t}$ of the signal $X$ at time $t$ given the information accumulated
from observing $Y$ in the interval $[0,t]$; that is, for $\varphi \in b%
\mathcal{B}(\mathbb{S})$, computing 
\begin{equation}
\pi _{t}(\varphi )=\mathbb{E}[\varphi (X_{t})\mid \mathcal{Y}_{t}].
\label{nfp}
\end{equation}
\end{definition}

There exists a suitable regularisation of the process $\pi =\{\pi _{t},\
t\geq 0\}$, so that $\pi _{t}$ is an optional $\mathcal{Y}_{t}$-adapted
probability measure-valued process for which (\ref{nfp}) holds almost surely%
\footnote{%
See Theorem 2.1 in \cite{bc}.}. In addition, since $\mathcal{Y}_{t}$ is
right-continuous, it follows that $\pi $ has a cadlag version (see Corollary
2.26 in \cite{bc}). In the following, we take $\pi $ to be this version.

In the following we deduce the evolution equation for $\pi $. A new measure
is constructed under which $Y$ becomes a Brownian motion and $\pi $ has a
representation in terms of an associated unnormalised version $\rho $. This $%
\rho $ is then shown to satisfy a linear evolution equation which leads to
the evolution equation for $\pi $ by an application of It\^{o}'s formula.

\subsection{Preliminary Results\label{sectionremark}}

\begin{definition}
We define $H^{1}\left( \mathbb{P}\right) $ to be the set of c\`{a}dl\`{a}g
real-valued $\mathcal{F}_{t}$-martingales $M=\{M_{t},\ t\geq 0\}$ such that
the associated process $M^{\ast }=\{M_{t}^{\ast },\ t\geq 0\}$ defined as $%
M_{t}^{\ast }:=\sup_{0\leq s\leq t}\left\vert M_{s}\right\vert $ for $t\geq
0 $ is a submartingale. In particular, $\mathbb{E}\left[ M_{t}\right]
<\infty $. \ for any $t\geq 0$.
\end{definition}

\begin{remark}
$H^{1}\left( \mathbb{P}\right) $ together with the distance function 
\begin{equation*}
d\left( M,N\right) :=\sum_{n=1}^{\infty }\frac{1}{2^{n}}\min \left( \mathbb{E%
}\left[ \left( M-N\right) _{n}^{\ast }\right] ,1\right)
\end{equation*}%
is a Fr\'{e}chet space with translation invariant metric. Suppose $\left(
W_{t}\right) _{t\geq 0}$ is an $%
\mathbb{R}
^{d}-$valued Brownian motion and $H=(H^{i})_{i=1}^{d}$ is an $\mathcal{F}%
_{t} $-adapted measurable $%
\mathbb{R}
^{d}$-valued process such that 
\begin{equation}
P\left( \int_{0}^{t}\left\vert H_{s}\right\vert ^{2}ds<\infty \right) =1.
\label{lsq}
\end{equation}%
Define $Z=\left( Z_{t}\right) _{t\geq 0}$ to be the exponential local
martingale\footnote{%
Here and later if $a=\left( a_{i}\right) _{i=1}^{d}\in 
\mathbb{R}
^{d}$, then $\left\vert a\right\vert ^{2}=\sum_{i=1}^{d}a_{i}^{2}$. Hence,
for example, in the expression for $Z$ from $\int_{0}^{t}\left\vert
H_{s}\right\vert ^{2}ds=$ $\sum_{i=1}^{d}\int_{0}^{t}\left( H_{s}^{i}\right)
^{2}ds$} 
\begin{equation*}
Z_{t}=\exp \left( \int_{0}^{t}H_{s}^{\top }dW_{s}-\frac{1}{2}%
\int_{0}^{t}\left\vert H_{s}\right\vert ^{2}ds\right) ,
\end{equation*}%
where $\int_{0}^{t}H_{s}^{\top
}dW_{s}:=\sum_{i=1}^{d}\int_{0}^{t}H_{s}^{i}dW_{s}^{i}$.
\end{remark}

\begin{lemma}[The $Z\log Z$ lemma]
\label{le:filterEqns:genuineMgale}For any $t\geq 0$ we have 
\begin{equation}
\sup_{\tau \in \mathcal{T}_{t}}\mathbb{E}\left[ Z_{\tau }\log Z_{\tau }%
\right] =\frac{1}{2}\mathbb{E}\left[ \int_{0}^{t}Z_{s}\left\vert
H_{s}\right\vert ^{2}ds\right] \in \left[ 0,\infty \right] ,  \label{id}
\end{equation}%
where $\mathcal{T}_{t}$ is the set of $\left( \mathcal{F}_{t}\right) $%
-stopping times bounded by $t$. If furthermore the terms in (\ref{id}) are
finite, then they are both equal to $\mathbb{E}\left[ Z_{t}\log Z_{t}\right]
.$ We also have%
\begin{equation}
\mathbb{E}\left[ Z_{t}^{\ast }\right] \leq \frac{e+1}{e-1}+\frac{e}{2\left(
e-1\right) }\mathbb{E}\left[ \int_{0}^{t}Z_{s}\left\vert H_{s}\right\vert
^{2}ds\right] \in \left[ 0,\infty \right] .  \label{logz}
\end{equation}
\end{lemma}

As an immediate consequence of this lemma we have

\begin{corollary}
\label{corgenuinemart} If the terms in (\ref{id}) are finite,then $\left(
Z_{t}\right) _{t\geq 0}$ is a genuine martingale, uniformly integrable over
any finite interval $\left[ 0,t\right] $, that belongs to $H^{1}\left( 
\mathbb{P}\right) .$
\end{corollary}

\begin{remark}
The first part of this corollary -- that $Z$ is a martingale if the terms in
(\ref{id}) are finite -- is not new. At the time of going to press J. Ruf
brought to the authors' attention that it is a consequence of the either of
two more general results: see Theorem 1 and Corollary 5 in \cite{ruf}. The
additional generality these results is in fact unnecessary for us. Since the
governing considerations of our presentation are those of economy and
self-sufficiency, we include a short proof of our result below.
\end{remark}

\begin{proof}
Let $L_{t}:=Z_{t}\log Z_{t}$ for $t\geq 0.$ If we assume that $\sup_{\tau
\in \mathcal{T}_{t}}\mathbb{E}\left[ L_{\tau }\right] $ is finite, then for
all $K\geq e$%
\begin{equation*}
\sup_{\tau \in \mathcal{T}_{t}}\mathbb{E}\left[ \left\vert Z_{\tau
}\right\vert 1_{\left\{ \left\vert Z_{\tau }\right\vert \geq K\right\} }%
\right] \leq \sup_{\tau \in \mathcal{T}_{t}}\frac{\mathbb{E}\left[
\left\vert Z_{\tau }\right\vert \log Z_{\tau }1_{\left\{ \left\vert Z_{\tau
}\right\vert \geq K\right\} }\right] }{\log K}\leq \frac{1}{\log K}\left(
\sup_{\tau \in \mathcal{T}_{t}}\mathbb{E}\left[ L_{\tau }\right]
+e^{-1}\right)
\end{equation*}%
the right hand side of which tends to zero as $K\rightarrow \infty .$ Hence
the family random variables 
\begin{equation*}
\left\{ Z_{\tau }:\tau \in \mathcal{T}_{t}\right\}
\end{equation*}%
is uniformly integrable. $Z$ is thus a martingale over $\left[ 0,t\right] $
and $L$, by Jensen's inequality, is a submartingale$.$ Using $P\left(
0<Z_{t}<\infty ,\ \text{for\ all}\ t<\infty \right) =1$ we have from It\^{o}%
's formula that 
\begin{equation*}
L_{t}=\underset{:=M_{t}}{\underbrace{\int_{0}^{t}\left( 1+\log Z_{s}\right)
Z_{s}H_{s}^{\top }dW_{s}}}+\underset{:=A_{t}}{\underbrace{\frac{1}{2}%
\int_{0}^{t}Z_{s}\left\vert H_{s}\right\vert ^{2}ds}},
\end{equation*}%
$M$ is a local martingale, hence the stopped process $M_{\cdot }^{\sigma
_{n}}:=M_{\cdot \wedge \sigma _{n}}$ is a martingale for some localising
sequence $0\leq \sigma _{n}\leq \sigma _{n+1}\uparrow \infty $ as $%
n\rightarrow \infty .$ For any $\tau \in \mathcal{T}_{t}$ we obtain%
\begin{equation*}
\mathbb{E}\left[ L_{\tau }^{\sigma _{n}}\right] =\mathbb{E}\left[ A_{\tau
}^{\sigma _{n}}\right] \leq \mathbb{E}\left[ L_{\tau }\right] \leq \mathbb{E}%
\left[ L_{t}\right] .
\end{equation*}%
Then, using Fatou's lemma\footnote{%
Which we may do since $L$ is bounded from below by $-e^{-1}.$} and the
monotone convergence theorem, we have 
\begin{equation*}
\mathbb{E}\left[ L_{\tau }\right] \leq \lim \inf_{n\rightarrow \infty }%
\mathbb{E}\left[ L_{\tau }^{\sigma _{n}}\right] =\lim \inf_{n\rightarrow
\infty }\mathbb{E}\left[ A_{\tau }^{\sigma _{n}}\right] =\mathbb{E}\left[
A_{\tau }\right] \leq \mathbb{E}\left[ L_{\tau }\right] \leq \mathbb{E}\left[
L_{t}\right] .
\end{equation*}%
Finally taking the supremum over $\tau \in \mathcal{T}_{t}$ yields%
\begin{equation*}
\mathbb{E}\left[ L_{t}\right] \leq \sup_{\tau \in \mathcal{T}_{t}}\mathbb{E}%
\left[ L_{\tau }\right] \leq \sup_{\tau \in \mathcal{T}_{t}}\mathbb{E}\left[
A_{\tau }\right] \leq \mathbb{E}\left[ A_{t}\right] \leq \sup_{\tau \in 
\mathcal{T}_{t}}\mathbb{E}\left[ L_{\tau }\right] \leq \mathbb{E}\left[ L_{t}%
\right] .
\end{equation*}%
and the equality (\ref{id}) holds in this case. If instead we know that $%
\mathbb{E}\left[ A_{t}\right] \,<\infty $, then by defining the sequence of
stopping times $\left( \tau _{n}\right) _{n=1}^{\infty },$ $0\leq \tau
_{n}\leq \tau _{n+1}$ \ by 
\begin{equation*}
\tau _{n}=\inf \left\{ t\geq 0:\left\vert Z_{t}\right\vert =\frac{1}{n}\text{
or }\left\vert Z_{t}\right\vert =n\right\}
\end{equation*}%
we have 
\begin{equation*}
\mathbb{E}\left[ M_{t\wedge \tau _{n}}^{2}\right] =\mathbb{E}\left[
\int_{0}^{t\wedge \tau _{n}}\left( 1+\log Z_{s}\right)
^{2}Z_{s}^{2}\left\vert H_{s}\right\vert ^{2}ds\right] \leq 2n^{2}\left(
1+\log n\right) ^{2}\mathbb{E}\left[ A_{t}\right] <\infty .
\end{equation*}%
From this we deduce that the stopped process $M_{\cdot }^{\tau
_{n}}:=M_{\cdot \wedge \tau _{n}}$ is a square-integrable martingale over $%
\left[ 0,t\right] .$ Combining this with the fact that $A_{t\wedge \tau
_{n}}\leq A_{t}$ yields%
\begin{equation*}
\mathbb{E}\left[ L_{\tau \wedge \tau _{n}}\right] =\mathbb{E}\left[ A_{\tau
\wedge \tau _{n}}\right] \leq \mathbb{E}\left[ A_{t}\right]
\end{equation*}%
for any $\tau \in \mathcal{T}_{t}.$ We notice that $\tau _{n}\uparrow \infty 
$ ,\ and hence $Z_{t\wedge \tau _{n}}\rightarrow Z_{t}$ a.s. as $%
n\rightarrow \infty .$ Then applying Fatou's lemma and taking the supremum
over all $\tau \in \mathcal{T}_{t}$ then gives that $\sup_{\tau \in \mathcal{%
T}_{t}}\mathbb{E}\left[ L_{\tau }\right] \leq \mathbb{E}\left[ A_{t}\right]
\,<\infty .$ The equality 
\begin{equation*}
\mathbb{E}\left[ L_{t}\right] =\mathbb{E}\left[ A_{t}\right] =\sup_{\tau \in 
\mathcal{T}_{t}}\mathbb{E}\left[ L_{\tau }\right] \in \lbrack 0,\infty )
\end{equation*}%
then follows from the first part of the proof. It is clear from the argument
that $A_{t}$ is not integrable if and only if $\sup_{\tau \in \mathcal{T}%
_{t}}\mathbb{E}\left[ L_{\tau }\right] =\infty .$

Turning attention to (\ref{logz}), we observe that the stopped process $%
L^{\tau _{n}}$ is a bounded submartingale, with a bounded martingale part
given by $M^{\tau _{n}}$. Hence, by a modification of a standard maximal
inequality (see page 52 in \cite{revuzyor}), we deduce that 
\begin{eqnarray*}
\mathbb{E}\left[ \left( Z^{\tau _{n}}\right) _{t}^{\ast }\right] &\leq &%
\frac{e+1}{e-1}+\frac{e}{e-1}\mathbb{E}\left[ L_{t\wedge \tau _{n}}\right] \\
&\leq &\frac{e+1}{e-1}+\frac{e}{e-1}\mathbb{E}\left[ A_{t\wedge \tau _{n}}%
\right] .
\end{eqnarray*}%
The proof is finished by an application of the monotone convergence theorem.
\end{proof}

\begin{remark}[A comparison with Kazamaki's criterion]
The criterion of finite transformed average energy: 
\begin{equation}
\mathbb{E}\left[ \int_{0}^{t}Z_{s}\left\vert H_{s}\right\vert ^{2}ds\right]
<\infty ,  \label{loc finite energy}
\end{equation}%
turns out to be a criterion for $Z$ to be a martingale that is independent
of Kazamaki's criterion -- and therefore of Novikov's criterion -- in the
sense that one is sometimes applicable when the other is not.

We give two examples to illustrate this. First, we can make use of a simple
example introduced in Revuz and Yor \cite{revuzyor} (page 366, Exercise
2.10.40) in which Kazamaki's criterion fails. Let $W$ be a scalar Brownian
motion with $W_{0}=0$ and set $H_{t}=\alpha W_{t}$ for some $\alpha >0.$
Recall that Kazamaki's criterion is that $\exp \left( \frac{1}{2}%
\int_{0}^{\cdot }H_{s}^{T}dW_{s}\right) $ should be a submartinagle. But, as
Revuz and Yor point out, $Z_{\cdot }=\exp \left( \alpha \int_{0}^{\cdot
}W_{s}dW_{s}-\frac{\alpha ^{2}}{2}\int_{0}^{\cdot }W_{s}^{2}ds\right) $ is a
true martingale on $[0,\infty )$ for all $\alpha ,$ but $\exp \left( \frac{%
\alpha }{2}\int_{0}^{t}W_{s}dW_{s}\right) $ ceases to be a submartingale for 
$t\geq \alpha ^{-1}.$ However, under the transformed probability measure $%
\mathbb{\tilde{P}}$, defined on the $\sigma $-ring $\cup _{t\geq 0}\mathcal{F%
}_{t}$ by 
\begin{equation*}
\left. \frac{d\mathbb{\tilde{P}}}{d\mathbb{P}}\right\vert _{\mathcal{F}%
_{t}}=Z_{t},
\end{equation*}%
$W$ is turned into a Gaussian semimartingale satisfying%
\begin{equation*}
W_{t}=\int_{0}^{t}\alpha W_{s}ds+B_{t}
\end{equation*}%
for some $\left( \left\{ \mathcal{F}_{t}\right\} _{t\geq 0},\mathbb{\tilde{P}%
}\right) $ Brownian motion $B.$ But $W$ can also be expressed as 
\begin{equation*}
W_{t}=\int_{0}^{t}e^{\alpha \left( t-s\right) }dB_{s}
\end{equation*}%
and then it is straightforward to show that for all $t\geq 0$%
\begin{equation*}
\mathbb{E}\left[ \int_{0}^{t}Z_{s}H_{s}^{2}ds\right] =\mathbb{\tilde{E}}%
\left[ \alpha ^{2}\int_{0}^{t}W_{s}^{2}ds\right] =\frac{1}{4}\left(
e^{2\alpha t}-2\alpha t-1\right) .
\end{equation*}%
Hence the transformed average energy condition is applicable in this case.

To give an example in the other direction, we construct a stopping time $S<1$
a.s., a continuous local martingale $X$ on $\left[ 0,1\right] $ with
quadratic variation 
\begin{equation*}
\left\langle X\right\rangle _{\cdot }=\int_{0}^{S\wedge \cdot }\frac{dr}{%
\left( 1-r\right) ^{2}}
\end{equation*}%
such that $e^{\frac{1}{2}X_{\cdot }}$ is a submartingale on $\left[ 0,1%
\right] $ and the transformed average energy satisfies 
\begin{equation*}
\mathbb{E}\left[ \int_{0}^{1}\frac{\zeta _{r}}{\left( 1-r\right) ^{2}}dr%
\right] =\mathbb{E}\left[ \int_{0}^{S}\frac{\zeta _{r}}{\left( 1-r\right)
^{2}}dr\right] =\infty ,
\end{equation*}%
where $\zeta $ is the exponential local martingale $\zeta _{t}=e^{X_{t}-%
\frac{1}{2}\left\langle X\right\rangle _{t}}.$ For this example, Kazamaki's
criterion implies that $\zeta $ is a martingale on the closed interval $%
\left[ 0,1\right] ,$ while the average energy condition fails to do so for $%
t=1.$

Suppose $W$ is an $\left\{ \mathcal{F}_{t}\right\} $-adapted Brownian
motion, null at zero, on a filtered probability space $\left( \Omega ,%
\mathcal{F},\left\{ \mathcal{F}_{t}\right\} _{t\geq 0},\mathbb{\tilde{P}}%
\right) $ and $N$ is an $\mathcal{F}_{0}-$measurable integer-valued random
variable, independent of $W$, with distribution under $\mathbb{P}$ given by $%
\mathbb{P}\left( N=n\right) =1/(n(n+1))$ for $n\in 
\mathbb{N}
.$ Introduce a sequence of stopping times%
\begin{equation*}
T_{n}:=\inf \left\{ t\geq 0:W_{t}=-1\text{ or }W_{t}=n\right\} ,
\end{equation*}%
with the convention that $T_{n}=\infty $ if this set is empty.

For each $n,$ $\mathbb{\tilde{P}}\left( T_{n}<\infty \right) =1$ and $%
W_{T_{n}\wedge \cdot }$ is a zero-mean bounded martingale with $\mathbb{%
\tilde{P}}\left( W_{T_{n}}=-1\right) =n/(n+1)$ and $\mathbb{\tilde{P}}\left(
W_{T_{n}}=n\right) =1/(n+1).$ Furthermore, by Jensen's inequality, $\exp
\left( -\frac{1}{2}W_{T_{n}\wedge \cdot }\right) $ is a positive
submartingale which is bounded uniformly in $n$ and $t$ by $e^{1/2}.$ We now
let $T=T_{N}.$ The process $e^{-\frac{1}{2}W_{T\wedge \cdot }}$ is also a
bounded submartingale since for all stopping times $R<S$ and for all $n$%
\begin{equation*}
\mathbb{E}_{\mathbb{\tilde{P}}}\left[ e^{-\frac{1}{2}W_{T\wedge R}};N=n%
\right] =\mathbb{E}_{\mathbb{\tilde{P}}}\left[ e^{-\frac{1}{2}W_{T_{n}\wedge
R}};N=n\right] \leq \mathbb{E}_{\mathbb{\tilde{P}}}\left[ e^{-\frac{1}{2}%
W_{T_{n}\wedge S}};N=n\right] =\mathbb{E}_{\mathbb{\tilde{P}}}\left[ e^{-%
\frac{1}{2}W_{T\wedge S}};N=n\right] .
\end{equation*}%
Now the strictly positive local martingale $\tilde{Z}_{\cdot
}:=e^{-W_{T\wedge \cdot }-\frac{1}{2}T\wedge \cdot }$ is bounded and hence
is a uniformly integrable martingale of the form $\tilde{Z}_{t}=\mathbb{E}_{%
\mathbb{\tilde{P}}}\left[ \left. \tilde{Z}_{T}\right\vert \mathcal{F}_{t}%
\right] .$ Let $\mathbb{P}$ be the probability measure which is equivalent
to $\mathbb{\tilde{P}}$ defined by $d\mathbb{P=}\tilde{Z}_{T}d\mathbb{\tilde{%
P}}$. Define on $\left[ 0,\infty \right] $ the process $Y:$%
\begin{eqnarray*}
Y_{t} &=&W_{T\wedge t}+T\wedge t\text{\ for }t\in \lbrack 0,\infty ),\text{
and} \\
Y_{\infty } &=&\left( W_{T}+T\right) 1_{\left\{ T<\infty \right\} }.
\end{eqnarray*}%
Girsanov's Theorem tells us that $Y$ is a local martingale under $\mathbb{P}$%
. Set $Z_{t}=\left( \tilde{Z}_{t}\right) ^{-1}.$ Then $Z_{\cdot
}=e^{Y_{\cdot }-T\wedge \cdot }$ on $[0,\infty ).$ We need to show that $e^{%
\frac{1}{2}Y_{\cdot }}$ is a submartingale under $\mathbb{P}$. But this
follows from the fact that for any finite stopping times $R<S,$%
\begin{eqnarray*}
\mathbb{E}_{\mathbb{P}}\left[ e^{\frac{1}{2}Y_{R}}\right] &=&\mathbb{E}_{%
\mathbb{\tilde{P}}}\left[ \tilde{Z}_{R}e^{\frac{1}{2}Y_{R}}\right] =\mathbb{E%
}_{\mathbb{\tilde{P}}}\left[ e^{-\frac{1}{2}W_{R}}\right] \\
&\leq &\mathbb{E}_{\mathbb{\tilde{P}}}\left[ e^{-\frac{1}{2}W_{S}}\right] =%
\mathbb{E}_{\mathbb{P}}\left[ e^{\frac{1}{2}Y_{S}}\right] ,
\end{eqnarray*}%
where we have used the fact that $e^{-\frac{1}{2}W_{\cdot }}$ is a
submartingale under $\mathbb{\tilde{P}}$. So Kazamaki's criterion allows us
to construct a probability measure $\mathbb{\bar{P}}$ such that, for all
stopping times $S$, $d\mathbb{\bar{P}=}Z_{S}d\mathbb{P}$ on $\mathcal{F}%
_{S}\cap \left\{ S<\infty \right\} .$ Since $Z_{S}=Z_{S\wedge T}=$ $\tilde{Z}%
_{S\wedge T}^{-1}$, and $\mathbb{P}\left( T<\infty \right) =\mathbb{\tilde{P}%
}\left( T<\infty \right) =1$ the measures $\mathbb{\bar{P}}$ and $\mathbb{%
\tilde{P}}$ coincide on $\mathcal{F}_{T}.$ Now the quadratic variation $%
\left\langle Y\right\rangle _{\cdot }=T\wedge \cdot $, and the integral in
the relevant transformed average energy condition is%
\begin{eqnarray*}
\mathbb{E}_{\mathbb{P}}\left[ \int_{0}^{T}Z_{s}ds\right] &=&\mathbb{E}_{%
\mathbb{P}}\left[ TZ_{T}\right] =\mathbb{E}_{\mathbb{\tilde{P}}}\left[ T%
\right] \\
&=&\mathbb{E}_{\mathbb{\tilde{P}}}\left[ W_{T}^{2}\right] \\
&=&\mathbb{\tilde{P}}\left( W_{T}=-1\right) +\mathbb{E}_{\mathbb{\tilde{P}}}%
\left[ W_{T}^{2};W_{T}\geq 1\right] \\
&=&\mathbb{\tilde{P}}\left( W_{T}=-1\right) +\sum_{n=1}^{\infty }\frac{n}{%
\left( n+1\right) ^{2}} \\
&=&\infty .
\end{eqnarray*}%
We now turn to the construction of $X$ and $\zeta .$ Let $\sigma
:[0,1]\rightarrow \left[ 0,\infty \right] $ be the time-change $\sigma
\left( t\right) =t\left( 1-t\right) ^{-1}.$ Let $X_{t}=Y_{\sigma \left(
t\right) }$ and $\zeta _{t}=Z_{\sigma \left( t\right) .}$\ Then $X,e^{\frac{1%
}{2}X}$ and $\zeta $ inherit, respectively, the local martingale, the
submartingale and the uniformly integrable martingale properties of $Y,e^{%
\frac{1}{2}Y}$ and $Z$ though with respect to the filtration $\left\{ 
\mathcal{F}_{\sigma \left( t\right) }\right\} _{0\leq t<1}.$ Set $S=T\left(
1+T\right) ^{-1};$ that is, $\sigma \left( S\right) =T.$ then the quadratic
variation%
\begin{equation*}
\,\left\langle X\right\rangle _{t}=T\wedge \sigma \left( t\right) =\frac{%
S\wedge t}{1-S\wedge t}=\int_{0}^{S\wedge t}\frac{dr}{\left( 1-r\right) ^{2}}%
.
\end{equation*}%
Furthermore%
\begin{equation*}
\mathbb{E}_{\mathbb{P}}\left[ \int_{0}^{S}\frac{\zeta _{r}}{\left(
1-r\right) ^{2}}dr\right] =\mathbb{E}_{\mathbb{P}}\left[ \int_{0}^{T}Z_{s}ds%
\right] =\infty .
\end{equation*}%
This completes the justification of the properties of the example.
\end{remark}

\begin{remark}
\label{remark}We record four observations:

\begin{enumerate}
\item The proof does not require the a priori assumption that $\mathbb{E}%
\left[ \int_{0}^{t}\left\vert H_{s}\right\vert ^{2}ds\right] <\infty .$
However observe that 
\begin{equation*}
\mathbb{E}\left[ \int_{0}^{t}Z_{s}\left\vert H_{s}\right\vert ^{2}ds\right] =%
\mathbb{E}\left[ \int_{0}^{t}\mathbb{E}\left[ Z_{t}|\mathcal{F}_{s}\right]
\left\vert H_{s}\right\vert ^{2}ds\right] =\mathbb{E}\left[
Z_{t}\int_{0}^{t}\left\vert H_{s}\right\vert ^{2}ds\right] .
\end{equation*}

\item If the Brownian motion $W$ is independent of $H$ then using the
sequence of stopping times $\left( \tau _{n}\right) _{n=1}^{\infty },$ $%
0\leq \tau _{n}\leq \tau _{n+1}$ \ by 
\begin{equation*}
\tau _{n}=\inf \left\{ t\geq 0:\left\vert H_{t}\right\vert \geq n\right\} ,
\end{equation*}%
we get that 
\begin{eqnarray*}
\mathbb{E}\left[ Z_{t\wedge \tau _{n}}|H\right] &=&\mathbb{E}\left[ \left.
\exp \left( \int_{0}^{t\wedge \tau _{n}}H_{s}^{\top }dW_{s}-\frac{1}{2}%
\int_{0}^{t\wedge \tau _{n}}\left\vert H_{s}\right\vert ^{2}ds\right)
\right\vert H\right] \\
&=&\exp \left( -\frac{1}{2}\int_{0}^{t\wedge \tau _{n}}\left\vert
H_{s}\right\vert ^{2}ds\right) \mathbb{E}\left[ \left. \exp \left(
\int_{0}^{t\wedge \tau _{n}}H_{s}^{\top }dW_{s}\right) \right\vert H\right]
=1.
\end{eqnarray*}%
In particular, the stopped process $Z^{\tau _{n}}$ is a martingale. Moreover 
\begin{equation*}
\mathbb{E}\left[ \left. \int_{0}^{t\wedge \tau _{n}}Z_{s}\left\vert
H_{s}\right\vert ^{2}ds\right\vert H\right] =\int_{0}^{t}\mathbb{E}\left[
\left. Z_{s\wedge \tau _{n}}\right\vert H\right] \left\vert H_{s\wedge \tau
_{n}}\right\vert ^{2}ds=\int_{0}^{t\wedge \tau _{n}}\left\vert
H_{s}\right\vert ^{2}ds.
\end{equation*}%
Hence, by an application of the monotone convergence theorem 
\begin{equation*}
\mathbb{E}\left[ \int_{0}^{t}Z_{s}\left\vert H_{s}\right\vert ^{2}ds\right] =%
\mathbb{E}\left[ \int_{0}^{t}\left\vert H_{s}\right\vert ^{2}ds\right] .
\end{equation*}%
By the same argument one can prove directly that $Z$ is a martingale under
the weaker condition (\ref{lsq}). This result is contained in Lemma 11.3.1
of \cite{kal}.

\item Assume that $\mathbb{E}\left[ A_{t}\right] \,<\infty $ for all $t\geq
0 $, then $\left( Z-1\right) $ is a zero-mean martingale and $\mathbb{E}%
\left[ \left( Z-1\right) _{t}^{\ast }\right] \,<1+$ $\mathbb{E}\left[
Z_{t}^{\ast }\right] \,<\infty $. Since \thinspace $\left\langle
Z-1\right\rangle _{t}=\int_{0}^{t}Z_{s}^{2}\left\vert H_{s}\right\vert
^{2}ds $ the Burkholder-Davis-Gundy inequalities gives%
\begin{equation*}
\mathbb{E}\left[ \left( \int_{0}^{t}Z_{s}^{2}\left\vert H_{s}\right\vert
^{2}ds\right) ^{1/2}\right] <\infty
\end{equation*}%
for all $t\geq 0.$

\item The finiteness of the transformed average energy does not imply that
the average energy itself is finite. The following example illustrates this.
Let $W=\left( W_{t}\right) _{0\leq t\leq 1}$ be a one-dimensional $\left( 
\mathcal{F}_{t}\right) _{0\leq t\leq 1}$ -adapted Brownian motion with $%
W_{0}=0$, and suppose that $\mathcal{F}_{0}$ carries a uniform $\left[ 0,1%
\right] $ random variable which is independent of $W.$ Then we will prove
there exists an $\left( \mathcal{F}_{t}\right) $-optional process $H=\left(
H_{t}\right) _{0\leq t\leq 1}$ such that the local martingale $Z$ given by 
\begin{equation*}
Z_{t}=\exp \left( \int_{0}^{t}H_{s}dW_{s}-\frac{1}{2}\int_{0}^{t}H_{s}^{2}ds%
\right)
\end{equation*}%
is a martingale on $\left[ 0,1\right] $ for which 
\begin{equation*}
\mathbb{E}\left[ \int_{0}^{1}Z_{s}H_{s}^{2}ds\right] <\infty \text{ and }%
\mathbb{E}\left[ \int_{0}^{1}H_{s}^{2}ds\right] =\infty .
\end{equation*}%
\qquad \qquad \qquad\ \ 

To construct $Z$ we will make use of the Gaussian martingale $%
B_{t}=\int_{0}^{t}\frac{1}{1-s}dW_{s}$ defined on $[0,1)$. We notice that $%
\left( 1-t\right) B_{t}$ is a Brownian bridge on $[0,1)$ and the related
process $V_{t}:=\exp \left[ B_{t}-\frac{t}{2\left( 1-t\right) }\right] $ is
just the martingale of densities on $\left( \mathcal{F}_{t}\right) _{0\leq
t<1}$ that turns $W$ into a Brownian bridge, cf. \cite{revuzyor}. $\ $But
the property we exploit is the existence of a Brownian motion $\bar{B}$ on $%
[0,\infty )$ such that $B_{t}=\bar{B}$ $_{\sigma \left( t\right) }$ wherein $%
\sigma \left( t\right) :=t\left( 1-t\right) ^{-1}.$ Let%
\begin{equation*}
X_{t}=\int_{0}^{t}\frac{V_{s}ds}{\left( 1-s\right) ^{2}}
\end{equation*}%
be defined on $\left[ 0,1\right] $ and introduce the sequence of stopping
times%
\begin{equation*}
T_{n}=\inf \left\{ t\geq 0:X_{t}=\frac{n\left( 1-t\right) }{n\left(
1-t\right) +t}\right\} .
\end{equation*}%
Since $X_{\cdot }$ is non-negative and increasing with $X_{0}=0$ and the
function $t\mapsto \frac{n\left( 1-t\right) }{n\left( 1-t\right) +t}$ is
strictly decreasing to $0,$ each $T_{n}$ is strictly less than one.
Furthermore the sequence $\left( T_{n}\right) _{n=1}^{\infty }$ increases to
a limit $T_{\infty }\leq 1.$ We need to prove that $\mathbb{P}\left(
T_{\infty }=1\right) >0.$ Using the fact that 
\begin{equation*}
\lim_{n\rightarrow \infty }\frac{n\left( 1-t\right) }{n\left( 1-t\right) +t}%
=1\text{ for all }t<1,
\end{equation*}%
it follows that $\mathbb{P}\left( T_{\infty }=1\right) =\mathbb{P}\left(
X_{1}<1\right) .$ However, 
\begin{eqnarray*}
X_{1} &=&\int_{0}^{1}\frac{1}{\left( 1-t\right) ^{2}}\exp \left[ B_{t}-\frac{%
t}{2\left( 1-t\right) }\right] \text{d}t \\
&=&\int_{0}^{\infty }\exp \left( \bar{B}_{t}-\frac{1}{2}s\right) \text{d}s
\end{eqnarray*}%
and it is result of Dufresne \cite{duf} (see also Yor \cite{yor}, page 15)
that this latter integral is distributed as twice the inverse of a standard
exponential random variable $Y$. In particular $\mathbb{P}\left(
X_{1}<1\right) =\mathbb{P}\left( Y>2\right) =e^{-2},$ from which it follows
that $\mathbb{P}\left( T_{\infty }=1\right) >0$ and, therefore, $\mathbb{E}%
\left[ \frac{T_{\infty }}{1-T_{\infty }}\right] =\infty .$ The monotone
convergence theorem implies that the sequence 
\begin{equation*}
m\left( n\right) :=\mathbb{E}\left[ \frac{T_{n}}{1-T_{n}}\right] \uparrow
\infty \text{ as }n\rightarrow \infty .
\end{equation*}%
Let $U$ be the uniform $\left[ 0,1\right] $ random variable on $\mathcal{F}%
_{0}$ referred to earlier. We can construct, as a measurable function of $U$%
, an integer random variable $N$ satisfying%
\begin{equation*}
\mathbb{E}\left[ m\left( N\right) \right] =\infty .
\end{equation*}%
If $T$ denotes the stopping time $T_{N}$ \ then $T<1,$ but also 
\begin{equation*}
\mathbb{E}\left[ \frac{T}{1-T}\right] =\mathbb{E}\left[ m\left( N\right) %
\right] =\infty .
\end{equation*}%
Finally we take $Z_{t}:=M_{t\wedge T}$ on $\left[ 0,1\right] $ and define $H$
to be the corresponding integrand 
\begin{equation*}
H_{t}=\left\{ 
\begin{array}{cc}
\left( 1-t\right) ^{-1} & \text{on }[0,T) \\ 
0 & \text{on }\left[ T,1\right]%
\end{array}%
\right. ,
\end{equation*}%
whereupon we have%
\begin{eqnarray*}
\mathbb{E}\left[ \int_{0}^{1}Z_{s}H_{s}^{2}ds\right] &=&\mathbb{E}\left[
X_{T}\right] =\mathbb{E}\left[ \frac{N\left( 1-T\right) }{N\left( 1-T\right)
+T}\right] <1,\text{ but} \\
\mathbb{E}\left[ \int_{0}^{1}H_{s}^{2}ds\right] &=&\mathbb{E}\left[
\int_{0}^{T}\frac{1}{\left( 1-t\right) ^{2}}dt\right] =\mathbb{E}\left[ 
\frac{T}{1-T}\right] =\infty
\end{eqnarray*}%
as required.
\end{enumerate}
\end{remark}

\begin{remark}
For any $K>0$, it is possible to decompose the local martingale $M$ as 
\begin{equation*}
M=M^{sq,K}+M^{d,K},
\end{equation*}%
where $M^{sq,K}$ is a locally square-integrable martingale with jumps
bounded by a constant $K$\ and $M^{d,K}$ is a purely discontinuous local
martingale with locally integrable total variation, with jumps greater than $%
K$, in such a manner that the quadratic variation process $\left[
M^{sq,K},M^{d,K}\right] $ is identically equal to 0. In what follows we will
discard the dependence on the constant $K$ in the notation for $M^{sq,K}$
and $M^{d,K}$. The first part of the statement is essentially Proposition
I.4.17 in \cite{js} while the second part follows from Theorem I.4.18 of the
same reference, i.e., from the classical decomposition of the local
martingale $M^{sq}$ into its continuous and purely discontinuous parts 
\begin{equation*}
M^{sq}=M^{sq,c}+M^{sq,d}.
\end{equation*}%
We have that 
\begin{equation*}
\left[ M^{sq},M^{d}\right] =\left[ M^{sq,c},M^{d}\right] +\left[
M^{sq,d},M^{d}\right] =0
\end{equation*}%
as $\left[ M^{sq,c},M^{d}\right] $ is null since it is the quadratic
variation between a continuous and a purely discontinuous martingale and
since $\left[ M^{sq,d},M^{d}\right] $ since it is the quadratic variation of
two purely discontinuous martingales with no jumps occurring at the same
time.
\end{remark}

For the following proposition, we introduce a positive $\mathcal{F}_{t}$%
-adapted cadlag semimartingale of the form 
\begin{equation*}
U_{t}=U_{0}+\int_{0}^{t}a_{s}ds+M_{t},
\end{equation*}%
where $a$ is a measurable $\mathcal{F}_{t}$-adapted process and $M$ is a
local $\mathcal{F}_{t}$-martingale null at zero\footnote{%
We will use the notation $\left[ \cdot ,\cdot \right] $\ to denote the
quadratic variation process of two local martingales. In addition, we will
use the notation $\left\langle \cdot ,\cdot \right\rangle $ to denote the
predictable quadratic variation process of two locally square integrable
martingales. The two processes coincide if one of the martingales is
continuous. For further details see, for example, Chapter 4 of \cite{rw}.}.
We also assume that $E\left[ U_{0}\right] <\infty $ and additionally that
the quadratic variation processes $\left\langle W^{i},M\right\rangle $ $%
i=1,...,m$ are absolutely continuous. In particular, there exists a
measurable $m$-dimensional $\mathcal{F}_{t}$-adapted process $%
N=(N^{i})_{i=1}^{m}$ such that 
\begin{equation*}
\left\langle W^{i},M\right\rangle _{t}=\int_{0}^{t}N_{s}^{i}ds,~~~\ \ t\geq
0,~~~~i=1,...,m.
\end{equation*}%
Moreover we will assume that there exists a positive constant $c$ such that 
\begin{equation}
\max \left( \left\vert a_{t}\right\vert ,\left\vert N_{t}\right\vert
^{2}\right) \leq c\max \left( U_{t},U_{t-}\right) ,~\ ~~~\ \ t\geq 0.
\label{control}
\end{equation}

\begin{proposition}
\label{localboundedness}Assume that the $\mathcal{F}_{t}$-adapted measurable
process $H=(H^{i})_{i=1}^{d}$ satisfies the inequality%
\begin{equation}
\left\vert H_{t}\right\vert ^{2}\leq c\max \left( U_{t},U_{t-}\right) ~~~~~\
\ t\geq 0.  \label{condonhandu}
\end{equation}%
Then the functions $t\rightarrow \mathbb{E}\left[ Z_{t}\left\vert
H_{t}\right\vert ^{2}\right] $, $t\rightarrow \mathbb{E}\left[ \left\vert
H_{t}\right\vert ^{2}\right] $ are locally bounded. In particular Lemma \ref%
{le:filterEqns:genuineMgale} allows us to deduce that the process $Z$ is a $%
H^{1}\left( \mathbb{P}\right) $ martingale.
\end{proposition}

\begin{proof}
Let $\left( T_{n}\right) _{n>0}$ be a localizing sequence of stopping times
such that the stopped process $\left( M_{T_{n}\wedge \cdot }^{sq}\right) $
is a square integrable martingale and the process $\left( M_{T_{n}\wedge
\cdot }^{d}\right) $ is a martingale with integrable total variation $%
Var\left( M^{d}\right) _{T_{n}\wedge \cdot }$. Now introduce the localizing
sequence $\left( S_{n}\right) _{n>0}$ where%
\begin{equation*}
S_{n}=\inf \left\{ t\geq 0\left\vert \max \left\{
Z_{t},\int_{0}^{t}\left\vert a_{s}\right\vert ds,U_{t-}\right\} \geq
n\right. \right\} \wedge T_{n}.
\end{equation*}%
Note that the left continuity of the processes listed in the inner brackets
implies that these processes, when stopped at \ $S_{n}$ are bounded by $n$.
Consider now the evolution equation for $ZU$, that is%
\begin{equation}
Z_{t}U_{t}=U_{0}+\int_{0}^{t}Z_{s}\left( a_{s}+H_{s}^{\top }N_{s}\right)
ds+\int_{0}^{t}Z_{s}\left( H_{s}^{\top }dW_{s}+dM_{s}^{sq}+dM_{s}^{d}\right)
.  \label{eqforzandu}
\end{equation}%
It follows that the expected value of $Z_{t}U_{t}$ is controlled by the sum
of the expected values of the six terms on the right hand side of (\ref%
{eqforzandu}). The stochastic integral terms in (\ref{eqforzandu}), when
stopped at $S_{n}$ become genuine martingales. They can be controlled as
follows:%
\begin{eqnarray*}
\mathbb{E}\left[ \left( \int_{0}^{t\wedge S_{n}}Z_{s}U_{s-}H_{s}^{\top
}dW_{s}\right) ^{2}\right] &=&\mathbb{E}\left[ \int_{0}^{t\wedge
S_{n}}Z_{s}^{2}U_{s-}^{2}\left\vert H_{s}\right\vert ^{2}ds\right] \\
&\leq &cn^{4}\mathbb{E}\left[ \int_{0}^{t\wedge S_{n}}\max \left(
U_{t},U_{t-}\right) ds\right] \leq cn^{5}t.
\end{eqnarray*}%
Here we have used the fact that, for all $t\geq 0,$ $\int_{0}^{t}P\left(
U_{s}\neq U_{s-}\right) ds=0.$ We also have that%
\begin{eqnarray*}
\mathbb{E}\left[ \left( \int_{0}^{t\wedge S_{n}}Z_{s}dM_{s}^{sq}\right) ^{2}%
\right] &=&\mathbb{E}\left[ \int_{0}^{t\wedge S_{n}}Z_{s}^{2}d\left\langle
M^{sq}\right\rangle _{s}\right] \leq n^{2}\mathbb{E}\left[ \left\langle
M^{sq}\right\rangle _{t\wedge S_{n}}\right] <\infty \\
\mathbb{E}\left[ \left\vert \int_{0}^{t\wedge
S_{n}}Z_{s}dM_{s}^{d}\right\vert \right] &\leq &n\mathbb{E}\left[ Var\left(
M^{d}\right) _{S_{n}\wedge t}\right] <\infty .
\end{eqnarray*}%
By taking the expectation of both sides in (\ref{eqforzandu}) stopped at $%
t\wedge S_{n}$, we deduce that 
\begin{eqnarray*}
\mathbb{E}\left[ Z_{t}U_{t}1_{\left\{ t\leq S_{n}\right\} }\right] &\leq &%
\mathbb{E}\left[ Z_{t\wedge S_{n}}U_{t\wedge S_{n}}\right] \\
&=&\mathbb{E}\left[ U_{0}\right] +\mathbb{E}\left[ \int_{0}^{t\wedge
S_{n}}Z_{s}\left( a_{s}+H_{s}^{\top }N_{s}\right) ds\right] \\
&\leq &\mathbb{E}\left[ U_{0}\right] +2c\mathbb{E}\left[ \int_{0}^{t}Z_{s}%
\max \left( U_{s},U_{s-}\right) 1_{\left\{ s\leq S_{n}\right\} }ds\right] \\
&\leq &\mathbb{E}\left[ U_{0}\right] +2c\int_{0}^{t}\mathbb{E}\left[
Z_{s}U_{s}1_{\left\{ s\leq S_{n}\right\} }\right] ds\leq e^{2ct}\mathbb{E}%
\left[ U_{0}\right] <\infty .
\end{eqnarray*}%
Note that the last inequality follows from Gronwall's lemma. Since $%
\lim_{n\rightarrow \infty }S_{n}=\infty $, we can then deduce by the
monotone convergence theorem that, for all $t>0,$%
\begin{equation}
\sup_{s\in \left[ 0,t\right] }\mathbb{E}\left[ Z_{s}U_{s}\right] \leq e^{2ct}%
\mathbb{E}\left[ U_{0}\right] .  \label{almost}
\end{equation}%
The local boundedness of $t\rightarrow \mathbb{E}\left[ Z_{t}\left\vert
H_{t}\right\vert ^{2}\right] $ follows from (\ref{condonhandu}) and (\ref%
{almost}). Similarly we show that for all $t>0,$%
\begin{equation*}
\sup_{s\in \left[ 0,t\right] }\mathbb{E}\left[ U_{s}\right] <\infty .
\end{equation*}%
by using the above argument with $H=0$ for all $t\geq 0$ (and therefore $%
Z_{t}=1$). This in turn implies the local boundedness of the functions $%
t\rightarrow \mathbb{E}\left[ \left\vert H_{t}\right\vert ^{2}\right] $.
\end{proof}

\section{Two Particular Cases\label{examples}}

\subsection{The signal is a jump-diffusion process}

We continue to assume that the observation process follows (\ref%
{eq:filterEq:observation}), and suppose that $X_{t}=(X_{t}^{i})_{i=1}^{d},$
for all $t\geq 0,$ is a cadlag solution of a $d$-dimensional stochastic
differential equation. This is driven by a triplet ($V$,$W,L)$ comprising a $%
p$-dimensional Brownian motion $V=(V^{j})_{j=1}^{p}$, the $m$-dimensional
Brownian motion $W=(W^{j})_{j=1}^{m}$ driving the observation process $Y,$
and an $\mathbb{R}^{r}$-valued L\'{e}vy process $L=(L^{j})_{j=1}^{r}$ with
no centred Gaussian component and with L\'{e}vy measure $F$ such that $%
F\left( \left\{ 0\right\} \right) =0.$ viz. 
\begin{equation}
X_{t}^{i}=X_{0}^{i}+\int_{0}^{t}f^{i}(X_{s-})\,\mathrm{d}s+\sum_{j=1}^{p}%
\int_{0}^{t}\sigma ^{ij}(X_{s-})\,\mathrm{d}V_{s}^{j}+\sum_{k=1}^{m}%
\int_{0}^{t}\bar{\sigma}^{ik}(X_{s-})\,\mathrm{d}W_{s}^{k}+\sum_{l=1}^{r}%
\int_{0}^{t}\tilde{\sigma}^{il}(X_{s-})\,\mathrm{d}L_{s}^{l},\quad \quad
\label{ss1}
\end{equation}%
for $i=1,\ldots ,d.$ We write $f=(f^{i})_{i=1}^{d}:\mathbb{R}^{d}\rightarrow 
\mathbb{R}^{d}$ , $\sigma =(\sigma ^{ij})_{i=1,\ldots ,d,j=1,\ldots ,p}:%
\mathbb{R}^{d}\rightarrow \mathbb{R}^{d\times p}$, $\bar{\sigma}=(\bar{\sigma%
}^{ij})_{i=1,\ldots ,d,j=1,\ldots ,m}:\mathbb{R}^{d}\rightarrow \mathbb{R}%
^{d\times m}$ and $\tilde{\sigma}=(\tilde{\sigma}^{ij})_{i=1,\ldots
,d,j=1,\ldots ,r}:\mathbb{R}^{d}\rightarrow \mathbb{R}^{d\times r}.$

We recall that a function $g:E\rightarrow F$ between two normed spaces $%
\left( E,\left\vert \left\vert \cdot \right\vert \right\vert _{E}\right) $
and $\left( F,\left\vert \left\vert \cdot \right\vert \right\vert
_{F}\right) $ has \textit{at most linear growth} if there exists $K<\infty $
such that 
\begin{equation*}
\left\vert \left\vert g\left( e\right) \right\vert \right\vert _{F}\leq
K\left( 1+\left\vert \left\vert e\right\vert \right\vert _{E}\right)
\end{equation*}%
for all $e\in E.$ We record the assumptions to be made on the coefficients
in the equation (\ref{ss1}).

\begin{condition}
\label{lip}We assume $f,$ $\sigma ,$ $\bar{\sigma}$ and $\tilde{\sigma}$ are
Borel and have at most linear growth.
\end{condition}

We will use $\mu $ to denote the Poisson random measure associated with $L$,
i.e. for every $t\geq 0$ \ and $A\in \mathcal{B}\left( \mathbb{R}%
^{r}\setminus \left\{ 0\right\} \right) $ the random measure $\mu \left(
t,\cdot \right) $ defined by 
\begin{equation*}
\mu \left( t,A\right) :=\sum_{0\leq s\leq t}1_{A}\left( \Delta L_{s}\right) .
\end{equation*}%
We let $\nu \left( t,\cdot \right) :=F\left( \cdot \right) t=\mathbb{E}\left[
\mu \left( 1,\cdot \right) \right] t,$ where $F\left( \cdot \right) $ is the
L\'{e}vy measure of $L,$ and denote the compensated measure by $\tilde{\mu}%
\left( t,A\right) =\mu \left( t,A\right) -\nu \left( t,A\right) .$ $L$ then
has a L\'{e}vy-Ito decomposition of the form%
\begin{equation}
L_{t}=at+\int_{0<\left\vert \rho \right\vert <1}\rho \tilde{\mu}\left( t,%
\mathrm{d}\rho \right) +\int_{\left\vert \rho \right\vert \geq 1}\rho \mu
\left( t,\mathrm{d}\rho \right) .  \label{LI}
\end{equation}

\begin{condition}
\label{sq int}Let $L=\left( L_{t}\right) _{t\geq 0}$ be a L\'{e}vy process
with L\'{e}vy measure $F.$We assume the square integrability condition 
\begin{equation*}
\int_{\left\vert \rho \right\vert \geq 1}\rho ^{2}F\left( \mathrm{d}\rho
\right) <\infty .
\end{equation*}
\end{condition}

\begin{remark}
\label{LIM}Whenever this condition is in force we have that 
\begin{equation}
\int_{\left\vert \rho \right\vert \geq 1}\rho F\left( \mathrm{d}\rho \right)
<\infty \text{ for every }t\geq 0,  \label{1st moment}
\end{equation}%
and hence the L\'{e}vy-Ito decomposition (\ref{LI}) may be rewritten as 
\begin{equation*}
L_{t}=bt+\int_{\mathbb{R}^{r}\setminus \left\{ 0\right\} }\rho \tilde{\mu}%
\left( t,\mathrm{d}\rho \right) ,
\end{equation*}%
where $b:=a-\int_{\left\vert \rho \right\vert \geq 1}\rho F\left( \mathrm{d}%
\rho \right) .$
\end{remark}

We continue to assume the dynamics for the observation process described in (%
\ref{eq:filterEq:observation}), and we now assume that (\ref{1st moment})
holds. We can restate this example in the language of Section \ref%
{sectionremark} by noticing that the process $\bar{X}=\left( X,Y\right) $ is
a solution to a martingale problem, with generator $A$ now given by%
\begin{eqnarray*}
A\phi \left( \bar{x}\right) &=&A\phi \left( x,y\right) \\
&=&\mathcal{L}\phi \left( x,y\right) +\,\int_{\mathbb{R}^{r}\setminus
\left\{ 0\right\} }\left[ \phi \left( x+\tilde{\sigma}(x)\eta ,y\right)
-\phi \left( x,y\right) -\sum_{i=1}^{d}\sum_{l=1}^{r}\frac{\partial \phi
\left( x,y\right) }{\partial x_{i}}\tilde{\sigma}^{il}(x)\eta ^{l}\right]
F\left( d\eta \right)
\end{eqnarray*}%
where 
\begin{equation*}
\mathcal{L}=\sum_{i=1}^{d}\tilde{f}^{i}(x)\,\frac{\partial }{\partial x_{i}}%
+\sum_{k=1}^{m}h^{k}(x,y)\,\frac{\partial }{\partial y_{k}}+\frac{1}{2}%
\sum_{i,j=1}^{d}\left( a^{ij}\left( x\right) +\bar{a}^{ij}\left( x\right)
\right) \frac{\partial ^{2}}{\partial x_{i}\partial x_{j}}\,+\frac{1}{2}%
\sum_{k=1}^{m}\frac{\partial ^{2}}{\partial x_{k}^{2}},
\end{equation*}%
with $\tilde{f}^{i}(x):=f^{i}(x)+b^{i},$ and $a=(a^{ij})_{i,j=1,\ldots ,d}:%
\mathbb{R}^{d}\rightarrow \mathbb{R}^{d\times d},\bar{a}=(\bar{a}%
^{ij})_{i,j=1,\ldots ,d}:\mathbb{R}^{d}\rightarrow \mathbb{R}^{d\times d}$
are the matrix-valued function defined respectively as%
\begin{equation*}
a^{ij}=\frac{1}{2}\sum_{k=1}^{p}\sigma ^{ik}\sigma ^{jk}=\frac{1}{2}\left(
\sigma \sigma ^{\top }\right) ^{ij}\text{ and }a^{ij}=\frac{1}{2}%
\sum_{k=1}^{m}\sigma ^{ik}\sigma ^{jk}=\frac{1}{2}\left( \bar{\sigma}\bar{%
\sigma}^{\top }\right) ^{ij}
\end{equation*}%
for all $i,j=1,\ldots ,d.$

To ensure the filtering equations described in Section \ref{section6} can be
applied to this example, we wish to establish that the functions $\mathbb{E}%
\left[ Z_{\cdot }\left\vert h\left( X_{\cdot }\right) \right\vert ^{2}\right]
$ and $\mathbb{E}\left[ \left\vert h\left( X_{\cdot }\right) \right\vert ^{2}%
\right] $\ are locally bounded.

\begin{corollary}
Assume the coefficients in (\ref{ss1}) satisfy Conditions \ref{lip} and that 
$\bar{\sigma}$ is uniformly bounded. Let $X_{t}=(X_{t}^{i})_{i=1}^{d}$
denote a $d-$dimensional jump-diffusion process which solves (\ref{ss1}) for
all $t\geq 0.$ Suppose the driving L\'{e}vy process $L$ has a L\'{e}vy
measure $F$ which satisfies $F\left( \left\{ 0\right\} \right) =0$ and has
no Gaussian part. Assume Condition \ref{sq int} and further suppose that $%
X_{0},V,W$ and $L$ are independent with $\mathbb{E}\left[ \left\vert
X_{0}\right\vert ^{2}\right] <\infty .$ Let $h:\mathbb{R}^{d}\rightarrow 
\mathbb{R}^{m}$ be any Borel measurable function for which there exists $K>0$
such that for all $x\in \mathbb{R}^{d}$ 
\begin{equation*}
\left\vert h\left( x\right) \right\vert \leq K\left( 1+\left\vert
x\right\vert \right) ,
\end{equation*}%
and let $Z=\left( Z_{t}\right) _{t\geq 0}$ be the positive local martingale
which solves $Z_{t}=1+\int_{0}^{t}Z_{s}h\left( X_{s}\right) ^{T}dW_{s}.$
Then $\mathbb{E}\left[ Z_{\cdot }\left\vert h\left( X_{\cdot }\right)
\right\vert ^{2}\right] $ and $\mathbb{E}\left[ \left\vert h\left( X_{\cdot
}\right) \right\vert ^{2}\right] $\ are locally bounded.
\end{corollary}

\begin{proof}
By exploiting Remark \ref{LIM} we can rewrite the SDE governing $X$ as 
\begin{equation*}
dX_{t}=\tilde{f}(X_{t-})\,\mathrm{d}t+\sigma (X_{t-})\,\mathrm{d}V_{t}+\bar{%
\sigma}(X_{t-})\,\mathrm{d}W_{t}+\int_{\mathbb{R}^{r}\setminus \left\{
0\right\} }\tilde{\sigma}(X_{t-})\rho \,\tilde{\mu}\left( \mathrm{d}t,%
\mathrm{d}\rho \right) ,
\end{equation*}%
where $\tilde{f}(x)=f(x)+b$ ($b$ is as given in Remark \ref{LIM}) is clearly
still locally Lipschitz. In order to apply the local boundedness lemma we
need to find a suitable process $U$ and the component processes in its
decomposition. To this end we let%
\begin{equation*}
U_{t}=1+\left\vert X_{t}\right\vert ^{2}.
\end{equation*}%
and use It\^{o}'s formula to obtain%
\begin{equation*}
U_{t}=1+\left\vert X_{0}\right\vert ^{2}+2\int_{0}^{t}X_{s-}^{T}\mathrm{d}%
X_{s}+\left[ X,X\right] _{t},
\end{equation*}%
where the quadratic variation $\left[ X,X\right] $ may be computed as%
\begin{eqnarray*}
\left[ X,X\right] _{t} &=&\int_{0}^{t}\text{tr}\left[ \sigma \left(
X_{s-}\right) ^{T}\sigma (X_{s-})+\bar{\sigma}\left( X_{s-}\right) ^{T}\bar{%
\sigma}\left( X_{s-}\right) \right] \mathrm{d}s+\int_{0}^{t}\int_{\mathbb{R}%
^{r}\setminus \left\{ 0\right\} }\text{tr}\left[ \tilde{\sigma}(X_{s-})\rho
\rho ^{T}\tilde{\sigma}(X_{s-})^{T}\right] \mu \left( \mathrm{d}s,\mathrm{d}%
\rho \right) \\
&=&\int_{0}^{t}\text{tr}\left[ \sigma \left( X_{s-}\right) ^{T}\sigma
(X_{s-})+\bar{\sigma}\left( X_{s-}\right) ^{T}\bar{\sigma}\left(
X_{s-}\right) \right] \mathrm{d}s+\sum_{0\leq s\leq t}\text{tr}\left[ \tilde{%
\sigma}(X_{s-})\Delta L_{s}\Delta L_{s}^{T}\tilde{\sigma}(X_{s-})^{T}\right]
.
\end{eqnarray*}%
Hence we may write $U$ as 
\begin{equation*}
U_{t}=U_{0}+\int_{0}^{t}a_{s}\mathrm{d}s+M_{t},
\end{equation*}%
where 
\begin{eqnarray*}
U_{0} &=&1+\left\vert X_{0}\right\vert ^{2} \\
a_{t} &=&2X_{t-}^{T}\tilde{f}(X_{t-})+\text{tr}\left[ \sigma \left(
X_{t-}\right) ^{T}\sigma (X_{t-})+\bar{\sigma}\left( X_{t-}\right) ^{T}\bar{%
\sigma}\left( X_{t-}\right) \right] \\
&&+\int_{0}^{t}\int_{\mathbb{R}^{r}\setminus \left\{ 0\right\} }\text{tr}%
\left[ \tilde{\sigma}(X_{s-})\rho \rho ^{T}\tilde{\sigma}(X_{s-})^{T}\right]
F\left( \mathrm{d}\rho \right) \mathrm{d}s
\end{eqnarray*}%
and $M$ is the local martingale%
\begin{equation*}
M_{t}=\int_{0}^{t}2X_{s-}^{T}\left[ \sigma (X_{s-})\,\mathrm{d}V_{s}+\bar{%
\sigma}(X_{s-})\,\mathrm{d}W_{s}\right] +\int_{0}^{t}\int_{\mathbb{R}%
^{r}\setminus \left\{ 0\right\} }\text{tr}\left[ \tilde{\sigma}(X_{s-})\rho
\rho ^{T}\tilde{\sigma}(X_{s-})^{T}\right] \tilde{\mu}\left( \mathrm{d}s,%
\mathrm{d}\rho \right) .
\end{equation*}

Condition \ref{lip} on $\tilde{f},\sigma ,\bar{\sigma}$ and $\tilde{\sigma}$
ensures the existence of $C>0$ such that 
\begin{equation*}
a_{t}\leq C\left( U_{t-}\vee U_{t}\right) ,
\end{equation*}%
moreover the boundedness of $\bar{\sigma}$ gives rise to the estimate 
\begin{equation*}
\left\vert \left\langle W,M\right\rangle _{t}^{\prime }\right\vert
=\left\vert \bar{\sigma}(X_{t-})X_{t-}\right\vert \leq K\left\vert
X_{t-}\right\vert \leq KU_{t-}^{1/2}.
\end{equation*}%
The result then follows from Proposition \ref{localboundedness}.\bigskip
\end{proof}

\begin{remark}
We may adapt this example to the case where $X$ be an $\left\{ \mathcal{F}%
_{t}\right\} $-adapted Markov process with values in a finite state space $I$
\end{remark}

\subsection{The \emph{change-detection} filtering problem.}

The following is a simple example with real-world applications which fits
within the above framework. The effect we try to capture is a sudden change
in the parameters of the model which describes the (stochastic) evolution of
the observed process. The following illustrates how such an effect might be
incorporated into the framework presented previously.

We assume that $Y$ is the real-valued process with dynamics%
\begin{equation*}
Y_{t}=\int_{0}^{t}\left( b_{0}+B1_{[T,\infty )}\left( s\right) \right)
Y_{s}ds+W_{t},
\end{equation*}%
where $W=\{W_{t},\ t\geq 0\}$ is a standard Brownian motion, $b_{0}$ a
constant and $B$ and $T$ independent random variables, which are also
independent of $W.\ $We also assume that $T\geq 0$ and that $\mathbb{E}\left[
e^{\lambda B^{2}}\right] <\infty $ for all $\lambda \in 
\mathbb{R}
.$ The process $X_{t}=\left( X_{t}^{1},X_{t}^{2}\right) $ is then defined by 
\begin{equation*}
X_{t}^{1}=B\text{ and }X_{t}^{2}=I_{[T,\infty )}(t),\quad \ t\geq 0,
\end{equation*}%
whereupon the process $\bar{X}_{t}=\left( X_{t}^{1},X_{t}^{2},Y_{t}\right) $
is adapted to the filtration%
\begin{equation*}
\left\{ \mathcal{F}_{t}\right\} _{t\geq 0}:=\left\{ \mathcal{\sigma }\left(
B,I_{[T,\infty )}(s),W_{s}:s\leq t\right) \vee \mathcal{N}\right\} _{t\geq
0},
\end{equation*}%
where $\mathcal{N}$ is the class of null sets of the completed $\sigma $%
-field $\mathcal{F}_{\infty }=$ $\mathcal{\bar{\sigma}}\left(
B,T,W_{s},s<\infty \right) .$ We introduce the uniquely defined cadlag $%
\left( \mathcal{B}\left( 
\mathbb{R}
\right) \times \mathcal{F}_{t}\right) -$optional processes%
\begin{eqnarray*}
\left( t,b,\omega \right) &\mapsto &H_{t}^{b}\left( \omega \right) =\left(
b_{0}+b1_{[T\left( \omega \right) ,\infty )}\left( t\right) \right)
Y_{t}^{b}\left( \omega \right) \\
\left( t,b,\omega \right) &\mapsto &Y_{t}^{b}\left( \omega \right)
=\int_{0}^{t}H_{s}^{b}\left( \omega \right) ds+W_{t}\left( \omega \right) ,
\end{eqnarray*}%
and set $Z_{t}^{b}:=\exp \left[ -\int_{0}^{t}H_{s}^{b}dW_{s}-\frac{1}{2}%
\int_{0}^{t}\left( H_{s}\right) ^{2}ds\right] .$ Notice that $B$ is $%
\mathcal{F}_{0}$-measurable, and hence the continuous process $\left(
Z_{t}^{B}\right) _{t\geq 0}$ is an $\left\{ \mathcal{F}_{t}\right\} $%
-adapted exponential local martingale. Again, as in the previous example, we
need to show that the functions $\mathbb{E}\left[ Z_{\cdot }^{B}\left(
H_{\cdot }^{B}\right) ^{2}\right] $ and $\mathbb{E}\left[ \left( H_{\cdot
}^{B}\right) ^{2}\right] $ are locally bounded. To do this, fix $b\in 
\mathbb{R}
$ and take the terms $U_{t}$ and $c$ in Proposition \ref{localboundedness}
to be%
\begin{equation*}
U_{t}=U_{t}^{b}:=1+\left( Y_{t}^{b}\right) ^{2}\text{ and }c=c\left(
b\right) :=4+\left( b_{0}+b\right) ^{2}.
\end{equation*}%
Then we may verify that the conditions of Proposition \ref{localboundedness}%
\ are satisfied. It is immediate from its proof that the conclusion of
Proposition \ref{localboundedness} can be strengthened to give the estimate%
\begin{equation*}
\max \left\{ \mathbb{E}\left[ Z_{t}^{b}\left( H_{t}^{b}\right) ^{2}\right] ,%
\mathbb{E}\left[ \left( H_{t}^{b}\right) ^{2}\right] \right\} \leq
e^{c\left( b\right) t}\mathbb{E}\left[ U_{0}^{b}\right] =e^{c\left( b\right)
t}.
\end{equation*}%
Consequently%
\begin{equation*}
\mathbb{E}\left[ Z_{t}^{B}\left( H_{t}^{B}\right) ^{2}\right] =\mathbb{E}%
\left[ \left. \mathbb{E}\left[ Z_{t}^{b}\left( H_{t}^{b}\right) ^{2}\right]
\right\vert _{b=B}\right] \leq \mathbb{E}\left[ e^{c\left( B\right) t}\right]
\end{equation*}%
and similarly 
\begin{equation*}
\mathbb{E}\left[ \left( H_{t}^{B}\right) ^{2}\right] \leq \mathbb{E}\left[
e^{c\left( B\right) t}\right] .
\end{equation*}%
These inequalities, together with the moment condition on $B$, give the
required result.

\section{The Change of Probability Measure Method}

We now have all the ingredients required for introducing a probability
measure with respect to which the process $Y$ becomes a Brownian motion. We
return to the set-up of Section 2. Define $Z=\left( Z_{t}\right) _{t\geq 0}$
to be the exponential local martingale 
\begin{equation*}
Z_{t}=\exp \left( -\int_{0}^{t}h\left( \bar{X}_{s}\right) ^{\top }dW_{s}-%
\frac{1}{2}\int_{0}^{t}\left\vert h\left( \bar{X}_{s}\right) \right\vert
^{2}ds\right) .
\end{equation*}%
The change of probability measure method consists in modifying the
probability measure on $\Omega $ by means of Girsanov's theorem. As we
require $Z$ to be a martingale in order to construct the change of measure,
Lemma \ref{le:filterEqns:genuineMgale} suggests the following as a suitable
condition to impose upon $h$, 
\begin{equation}
\mathbb{E}\left[ \int_{0}^{t}Z_{s}\left\Vert h(\bar{X}_{s})\right\Vert ^{2}\,%
\mathrm{d}s\right] <\infty ,\quad \forall t>0.  \label{condonh}
\end{equation}

Let us assume that (\ref{condonh}) holds. Then, by Lemma \ref%
{le:filterEqns:genuineMgale}, $Z$ is a true martingale. Let $\mathbb{\tilde{P%
}}$ be the probability measure defined on the field $\bigcup_{0\leq t<\infty
}\mathcal{F}_{t}$ that is specified by its Radon--Nikodym derivative $Z_{t}$
on each $\mathcal{F}_{t}$ with respect to the corresponding trace of $%
\mathbb{P}$; that is, for each $t\geq 0$:%
\begin{equation*}
\left. \frac{\mathrm{d}\tilde{\mathbb{P}}}{\mathrm{d}\mathbb{P}}\right\vert
_{\mathcal{F}_{t}}=Z_{t}.
\end{equation*}%
$\mathbb{\tilde{P}}$ restricted to each $\mathcal{F}_{t}$ is equivalent to $%
\mathbb{P}$ since $Z_{t}$ is a positive random variable\footnote{%
Note that we have not defined $\mathbb{\tilde{P}}$ on $\mathcal{F}_{\infty }$%
, where $\mathcal{F}_{\infty }=\bigvee_{t=0}^{\infty }\mathcal{F}_{t}=\sigma
\left( \bigcup_{0\leq t<\infty }\mathcal{F}_{t}\right) .$}.

Let $\tilde{Z}=\{\tilde{Z}_{t},\ t\geq 0\}$ be the process defined as $%
\tilde{Z}_{t}=Z_{t}^{-1}$ for $t\geq 0$. Under $\tilde{\mathbb{P}}$, $\tilde{%
Z}_{t}$ satisfies the following stochastic differential equation, 
\begin{equation}
\mathrm{d}\tilde{Z}_{t}=\sum_{i=1}^{m}\tilde{Z}_{t}h^{i}(X_{t})\,\mathrm{d}%
Y_{t}^{i}  \label{ztilde1}
\end{equation}%
and since $\tilde{Z}_{0}=1$, 
\begin{equation}
\tilde{Z}_{t}=\exp \left( \sum_{i=1}^{m}\int_{0}^{t}h^{i}(X_{s})\,\mathrm{d}%
Y_{s}^{i}-\frac{1}{2}\sum_{i=1}^{m}\int_{0}^{t}h^{i}(X_{s})^{2}\,\mathrm{d}%
s\right) ,  \label{ztilde2}
\end{equation}%
then $\tilde{\mathbb{E}}[\tilde{Z}_{t}]=\mathbb{E}[\tilde{Z}_{t}Z_{t}]=1$.
So $\tilde{Z}$ is an $\mathcal{F}_{t}$-adapted martingale under $\tilde{%
\mathbb{P}}$ and%
\begin{equation*}
\left. \frac{\mathrm{d}\mathbb{P}}{\mathrm{d}\tilde{\mathbb{P}}}\right\vert
_{\mathcal{F}_{t}}=\tilde{Z}_{t}\quad \mathrm{\ for\ }t\geq 0.
\end{equation*}%
$\mathbb{P}$ and $\mathbb{\tilde{P}}$ are therefore equivalent on each $%
\mathcal{F}_{t}$ for $t\geq 0$.

\begin{proposition}
\label{prop:filterEqns:YisBM} If condition (\ref{condonh}) is satisfied,
then under $\tilde{\mathbb{P}}$ the observation process $Y$ is a Brownian
motion. Let $\varphi \in \mathcal{D}(A)$ have bounded derivatives in the $y$%
-direction, and let $\tilde{M}^{\varphi }$ denote the semimartingale 
\begin{equation*}
\tilde{M}_{t}^{\varphi }:=M_{t}^{\varphi }+\int_{0}^{t}\sum_{i=1}^{m}\left(
h^{i}B^{i}\varphi +\frac{\partial \varphi }{\partial y_{i}}\right) \left( 
\bar{X}_{t}\right) \mathrm{d}s.
\end{equation*}%
Then the stochastic integral $\int_{0}^{\cdot }\tilde{Z}_{s}d\tilde{M}%
_{s}^{\varphi }$ is a zero-mean martingale under $\tilde{\mathbb{P}}$.
\end{proposition}

\begin{proof}
Lemma \ref{le:filterEqns:genuineMgale}, together with condition \ref{condonh}%
, ensures that $Z$ is a martingale (under $\mathbb{P}$) and that $\mathbb{%
\tilde{P}}$ is a probability measure on each $\mathcal{F}_{t}.$That $Y$
becomes a Brownian motion under $\mathbb{\tilde{P}}$ is an immediate
consequence of Girsanov's theorem. For brevity, let $\beta $ denote the
process defined by 
\begin{equation*}
\beta _{t}:=\sum_{i=1}^{m}\left( h^{i}B^{i}\varphi +\frac{\partial \varphi }{%
\partial y_{i}}\right) \left( \bar{X}_{t}\right) ;
\end{equation*}%
then $\tilde{M}_{t}^{\varphi }$ can be expressed as $M_{t}^{\varphi
}+\int_{0}^{t}\beta _{s}\mathrm{d}s.$ It also follows from (\ref{covariance}%
) and the definition of $\tilde{Z}$ that $\left\langle M^{\varphi },\tilde{Z}%
\right\rangle _{t}=\int_{0}^{t}\tilde{Z}_{s}\beta _{s}\mathrm{d}s.$ But by It%
\^{o}'s integration-by-parts formula 
\begin{eqnarray}
\tilde{Z}_{t}M_{t}^{\varphi } &=&\int_{0}^{t}M_{s}^{\varphi }\mathrm{d}%
\tilde{Z}_{s}+\int_{0}^{t}\tilde{Z}_{s}\mathrm{d}M_{s}^{\varphi
}+\left\langle M^{\varphi },\tilde{Z}\right\rangle _{t}  \notag \\
&=&\int_{0}^{t}M_{s}^{\varphi }\mathrm{d}\tilde{Z}_{s}+\int_{0}^{t}\tilde{Z}%
_{s}\left( \mathrm{d}M_{s}^{\varphi }+\beta _{s}\mathrm{d}s\right)  \notag \\
&=&\int_{0}^{t}M_{s}^{\varphi }\mathrm{d}\tilde{Z}_{s}+\int_{0}^{t}\tilde{Z}%
_{s}\mathrm{d}\tilde{M}_{s}^{\varphi }.  \label{deco}
\end{eqnarray}%
However $M^{\varphi }$ being a martingale under $\mathbb{\tilde{P}}$ implies
that $\tilde{Z}M^{\varphi }$ is a martingale under $\mathbb{\tilde{P}}$, and
the first integral on the right-hand side is a martingale under $\mathbb{%
\tilde{P}}$ because $M^{\varphi }$ is bounded on finite intervals and $%
\tilde{Z}$ itself is a martingale. The conclusion of the proposition follows.
\end{proof}

\begin{remark}
Since $\mathbb{P}$ and $\mathbb{\tilde{P}}$ are absolutely continuous with
respect to each other, they have the same class of null sets $\mathcal{N}$
and therefore the (augmented) observation filtration is the same both under $%
\mathbb{P}$ and $\mathbb{\tilde{P}}$. Since $Y$ is a Brownian motion under $%
\mathbb{\tilde{P}}$ it follows that the filtration $\{\mathcal{Y}_{t},\
t\geq 0\}$ is right-continuous both under $\mathbb{P}$ and $\mathbb{\tilde{P}%
}$. To put it differently, $\{\mathcal{Y}_{t},\ t\geq 0\}$ satisfies the
usual conditions both under $\mathbb{P}$ and under $\mathbb{\tilde{P}}$.
\end{remark}

The following proposition is a consequence of the Brownian motion property
of the process $Y$ under $\tilde{\mathbb{P}}$.

\begin{proposition}
\label{prop:filterEqns:p4} Let $U$ be an integrable $\mathcal{F}_t$%
-measurable random variable. Then we have 
\begin{equation}
\tilde{\mathbb{E}}[U\mid \mathcal{Y}_{t}]=\tilde{\mathbb{E}}[U\mid \mathcal{Y%
}].  \label{yty}
\end{equation}
\end{proposition}

\begin{proof}
Let us denote by 
\begin{equation*}
\mathcal{Y}_{t}^{\prime }=\sigma (Y_{t+u}-Y_{t};\ u\geq 0);
\end{equation*}%
then $\mathcal{Y}=\sigma (\mathcal{Y}_{t},\mathcal{Y}_{t}^{\prime })$. Under
the probability measure $\tilde{\mathbb{P}}$ the $\sigma $-algebra $\mathcal{%
Y}_{t}^{\prime }\subset \mathcal{Y}$ is independent of $\mathcal{F}_{t}$
because $Y$ is an $\mathcal{F}_{t}$-adapted Brownian motion. Hence since $U$
is $\mathcal{F}_{t}$-adapted using the property (f) of conditional
expectation 
\begin{equation*}
\tilde{\mathbb{E}}[U\mid \mathcal{Y}_{t}]=\tilde{\mathbb{E}}[U\mid \sigma (%
\mathcal{Y}_{t},\mathcal{Y}_{t}^{\prime })]=\tilde{\mathbb{E}}[U\mid 
\mathcal{Y}].
\end{equation*}
\end{proof}

\section{Unnormalised Conditional Distribution}

In this section we first prove the Kallianpur--Striebel formula and use this
to define the unnormalized conditional distribution process. The notation $%
\tilde{\mathbb{P}}(\mathbb{P})$-a.s. below means that the result holds both $%
\mathbb{\tilde{P}}$-a.s. and $\mathbb{P}$-a.s. We only need to show that it
holds true in the first sense since $\tilde{\mathbb{P}}$ and $\mathbb{P}$
are equivalent probability measures.

\begin{proposition}[\textbf{Kallianpur--Striebel}]
\label{prop:filterEqns:kallstrie}Assume that condition (\ref{condonh})
holds. For every $\varphi \in $\smallskip $b\mathcal{B}(\mathbb{S})$, for
fixed $t\in \lbrack 0,\infty )$, 
\begin{equation}
\pi _{t}(\varphi )=\frac{\tilde{\mathbb{E}}[\tilde{Z}_{t}\varphi (X_{t})\mid 
\mathcal{Y}]}{\tilde{\mathbb{E}}[\tilde{Z}_{t}\mid \mathcal{Y}]}\quad \quad 
\tilde{\mathbb{P}}(\mathbb{P})\text{-}\mathrm{a.s.}  \label{kallstrie}
\end{equation}
\end{proposition}

\begin{proof}
It is clear from the definition that $\tilde{Z}_{t}>0$ $\tilde{\mathbb{P}}(%
\mathbb{P})$-a.s. as a consequence of which $\tilde{\mathbb{E}}[\tilde{Z}%
_{t}\mid \mathcal{Y}]>0$ $\mathbb{P}$-a.s. and the right-hand side of (\ref%
{kallstrie}) is well defined. It suffices to show that 
\begin{equation*}
\pi _{t}(\varphi )\tilde{\mathbb{E}}[\tilde{Z}_{t}\mid \mathcal{Y}_{t}]=%
\tilde{\mathbb{E}}[\tilde{Z}_{t}\varphi (X_{t})\mid \mathcal{Y}_{t}]\quad
\quad \tilde{\mathbb{P}}\text{-a.s.}
\end{equation*}%
As both the left- and right-hand sides of this equation are $\mathcal{Y}_{t}$%
-measurable, this is equivalent to showing that for any bounded $\mathcal{Y}%
_{t}$-measurable random variable $b$, 
\begin{equation*}
\tilde{\mathbb{E}}[\pi _{t}(\varphi )\tilde{\mathbb{E}}[\tilde{Z}_{t}\mid 
\mathcal{Y}_{t}]b]=\tilde{\mathbb{E}}[\tilde{\mathbb{E}}[\tilde{Z}%
_{t}\varphi (X_{t})\mid \mathcal{Y}_{t}]b].
\end{equation*}%
A consequence of the definition of the process $\pi _{t}$ is that $\pi
_{t}\varphi =\mathbb{E}[\varphi (X_{t})\mid \mathcal{Y}_{t}]$ $\tilde{%
\mathbb{P}}$-a.s., so from the definition of Kolmogorov conditional
expectation 
\begin{equation*}
\mathbb{E}\left[ \pi _{t}(\varphi )b\right] =\mathbb{E}\left[ \varphi
(X_{t})b\right] .
\end{equation*}%
Writing this under the measure $\tilde{\mathbb{P}}$, 
\begin{equation*}
\tilde{\mathbb{E}}\left[ \pi _{t}(\varphi )b\tilde{Z}_{t}\right] =\tilde{%
\mathbb{E}}\left[ \varphi (X_{t})b\tilde{Z}_{t}\right] .
\end{equation*}%
Since the function $b$ is $\mathcal{Y}_{t}$-measurable, by the tower
property of the conditional expectation, 
\begin{equation*}
\tilde{\mathbb{E}}\left[ \pi _{t}(\varphi )\tilde{\mathbb{E}}[\tilde{Z}%
_{t}\mid \mathcal{Y}_{t}]b\right] =\tilde{\mathbb{E}}\left[ \tilde{\mathbb{E}%
}[\varphi (X_{t})\tilde{Z}_{t}\mid \mathcal{Y}_{t}]b\right]
\end{equation*}%
which proves that the result holds $\tilde{\mathbb{P}}$-a.s.
\end{proof}

Let $\zeta =\{\zeta _{t},\ t\geq 0\}$ be the process defined by 
\begin{equation}
\zeta _{t}=\tilde{\mathbb{E}}[\tilde{Z}_{t}\mid \mathcal{Y}_{t}],
\label{eq:filterEqns:zeta}
\end{equation}%
then as $\tilde{Z}_{t}$ is an $\mathcal{F}_{t}$-martingale under $\tilde{%
\mathbb{P}}$ and $\mathcal{Y}_{s}\subseteq \mathcal{F}_{s}$, it follows that
for $0\leq s<t$, 
\begin{equation*}
\tilde{\mathbb{E}}[\zeta _{t}\mid \mathcal{Y}_{s}]=\tilde{\mathbb{E}}[\tilde{%
Z}_{t}|\mathcal{Y}_{s}]=\tilde{\mathbb{E}}\left[ \tilde{\mathbb{E}}[\tilde{Z}%
_{t}\mid \mathcal{F}_{s}]\mid \mathcal{Y}_{s}\right] =\tilde{\mathbb{E}}[%
\tilde{Z}_{s}\mid \mathcal{Y}_{s}]=\zeta _{s}.
\end{equation*}%
Therefore by Doob's regularization theorem (see Rogers and Williams, \cite[%
Theorem II.67.7]{rw}) since the filtration $\mathcal{Y}_{t}$ satisfies the
usual conditions we can choose a c\`{a}dl\`{a}g version of $\zeta _{t}$
which is a $\mathcal{Y}_{t}$-martingale. In what follows, assume that $%
\{\zeta _{t},t\geq 0\}$ has been chosen to be such a version. 
Given such a $\zeta $, Proposition \ref{prop:filterEqns:kallstrie} suggests
the following definition.

\begin{definition}
\label{def:filterEqns:ucd} Define the \emph{unnormalised conditional
distribution} of $X$ to be the measure-valued process $\rho =\{\rho _{t},\
t\geq 0\}\ $given by $\rho _{t}=\zeta _{t}\pi _{t}$ for any $t\geq 0.$
\end{definition}

\begin{lemma}
\label{le:filterEqns:rhoCadlag}The process $\{\rho _{t},\ t\geq 0\}$ is c%
\`{a}dl\`{a}g and $\mathcal{Y}_{t}$-adapted. Furthermore, for any $t\geq 0$, 
\begin{equation}
\rho _{t}(\varphi )=\tilde{\mathbb{E}}\left[ \tilde{Z}_{t}\varphi
(X_{t})\mid \mathcal{Y}_{t}\right] \quad \quad \tilde{\mathbb{P}}(\mathbb{P})%
\text{-a.s.}  \label{eq:filterEqns:altRho}
\end{equation}
\end{lemma}

\begin{proof}
Both $\pi _{t}(\varphi )$ and $\zeta _{t}$ are $\mathcal{Y}_{t}$-adapted. By
construction $\{\zeta _{t},\ t\geq 0\}$ is also c\`{a}dl\`{a}g. We know that 
$\{\pi _{t},\ t\geq 0\}$ is c\`{a}dl\`{a}g and $\mathcal{Y}_{t}$-adapted;
therefore the process $\{\rho _{t},t\geq 0\}$ is also c\`{a}dl\`{a}g and $%
\mathcal{Y}_{t}$-adapted.

For the second part, from Proposition \ref{prop:filterEqns:p4} and
Proposition \ref{prop:filterEqns:kallstrie} it follows that 
\begin{equation*}
\pi _{t}(\varphi )\tilde{\mathbb{E}}[\tilde{Z}_{t}\mid \mathcal{Y}_{t}]=%
\tilde{\mathbb{E}}[\tilde{Z}_{t}\varphi (X_{t})\mid \mathcal{Y}_{t}]\quad
\quad \tilde{\mathbb{P}}\text{-a.s.},
\end{equation*}%
From (\ref{eq:filterEqns:zeta}), $\tilde{\mathbb{E}}[\tilde{Z}_{t}\mid 
\mathcal{Y}_{t}]=\zeta _{t}$ a.s. from which the result follows.
\end{proof}

\begin{corollary}
Assume that condition (\ref{condonh}) holds. For every $\varphi \in B(%
\mathbb{S})$, 
\begin{equation}  \label{eq:filterEqns:corKS}
\pi_t(\varphi)=\frac{\rho_t(\varphi)}{ \rho_t(\mathbf{1})} \quad\quad
\forall t \in [0,\infty)\quad \quad \tilde{\mathbb{P}}(\mathbb{P}) \text{%
-a.s.}
\end{equation}
\end{corollary}

\begin{proof}
It is clear from Definition \ref{def:filterEqns:ucd} that $\zeta _{t}=\rho
_{t}(\mathbf{1})$. The result then follows immediately.
\end{proof}

The Kallianpur--Striebel formula explains the usage of the term \emph{%
unnormalised} in the definition of $\rho _{t}$ as the denominator $\rho _{t}(%
\mathbf{1})$ can be viewed as the normalising factor.

\begin{lemma}
\begin{description}
\item[i.] Let $\{u_{t},\ t\geq 0\}$ be an $\mathcal{F}_{t}$-progressively
measurable process such that for all $t\geq 0$, we have 
\begin{equation}
\tilde{\mathbb{E}}\left[ \left( \int_{0}^{t}u_{s}^{2}\,\mathrm{d}s\right)
^{1/2}\right] <\infty ;  \label{scondu}
\end{equation}%
then, for all $t\geq 0$, and $j=1,\ldots ,m$, we have%
\begin{equation}
\tilde{\mathbb{E}}\left[ \left. \int_{0}^{t}u_{s}\,\mathrm{d}%
Y_{s}^{j}\,\right\vert \,\mathcal{Y}\right] =\int_{0}^{t}\tilde{\mathbb{E}}%
[u_{s}\mid \mathcal{Y}]\,\mathrm{d}Y_{s}^{j}.  \label{condU}
\end{equation}

\item[ii.] Let $\tilde{M}^{\varphi }$ be as defined in Proposition \ref%
{prop:filterEqns:YisBM}. Then for all $t\geq 0$ 
\begin{equation}
\tilde{\mathbb{E}}\left[ \left. \int_{0}^{t}\tilde{Z}_{s}\,\mathrm{d}\tilde{M%
}_{s}^{\varphi }\,\right\vert \,\mathcal{Y}\right] =\sum_{j=1}^{m}%
\int_{0}^{t}\tilde{\mathbb{E}}\left[ \left. \left( B^{j}\varphi +\frac{%
\partial \varphi }{\partial y_{j}}\right) \left( \bar{X}_{s}\right) \tilde{Z}%
_{s}\,\right\vert \,\mathcal{Y}\right] \mathrm{d}Y_{s}^{j},  \label{condu''}
\end{equation}
\end{description}
\end{lemma}

\begin{proof}
\begin{description}
\item 

\item[i.] To deduce the results we introduce the set of uniformly bounded
test random variables%
\begin{equation}
S_{t}=\left\{ \varepsilon _{t}=\exp \!\left( i\int_{0}^{t}r_{s}^{\top }\,%
\mathrm{d}Y_{s}+\frac{1}{2}\int_{0}^{t}\Vert r_{s}\Vert ^{2}\,\mathrm{d}%
s\right) :r\in L^{\infty }\left( [0,t],\mathbb{R}^{m}\right) \right\} .
\label{class}
\end{equation}%
Then $S_{t}$ is a total set. That is, if $a\in L^{1}(\Omega ,\mathcal{Y}_{t},%
\tilde{\mathbb{P}})$ and $\tilde{\mathbb{E}}[a\varepsilon _{t}]=0$, for all $%
\varepsilon _{t}\in S_{t}$, then $a=0$ $\tilde{\mathbb{P}}$-a.s. For a proof
of this result see, for example, Lemma B.39 page 355 in Bain and Crisan \cite%
{bc}. In addition, if $\varepsilon _{t}\in S_{t}$ $,$ then 
\begin{equation*}
\varepsilon _{t}=1+\int_{0}^{t}i\varepsilon _{s}r_{s}^{\top }\,\mathrm{d}%
Y_{s}.
\end{equation*}%
From condition (\ref{scondu}) it follows, by Burkholder-Davis-Gundy's
inequalities that both processes $t\rightarrow \int_{0}^{t}u_{s}\,\mathrm{d}%
Y_{s}^{j}$ and $t\rightarrow \int_{0}^{t}\tilde{\mathbb{E}}\left[ \left.
u_{s}\,\right\vert \,\mathcal{Y}\right] \,\mathrm{d}Y_{s}^{j}$ belong to $%
H^{1}(\mathbb{\tilde{P})}$.\ In particular they are zero-mean martingales.
We observe the following sequence of identities 
\begin{align*}
\tilde{\mathbb{E}}\left[ \varepsilon _{t}\tilde{\mathbb{E}}\left[ \left.
\int_{0}^{t}u_{s}\,\mathrm{d}Y_{s}^{j}\,\right\vert \,\mathcal{Y}\right] %
\right] & =\tilde{\mathbb{E}}\left[ \varepsilon _{t}\int_{0}^{t}u_{s}\,%
\mathrm{d}Y_{s}^{j}\right] \\
& =\tilde{\mathbb{E}}\left[ \int_{0}^{t}u_{s}\,\mathrm{d}Y_{s}^{j}\right] +%
\tilde{\mathbb{E}}\left[ \int_{0}^{t}i\varepsilon _{s}r_{s}^{j}u_{s}\,%
\mathrm{d}s\right] \\
& =\tilde{\mathbb{E}}\left[ \left. \tilde{\mathbb{E}}\left[
\int_{0}^{t}i\varepsilon _{s}r_{s}^{j}u_{s}\,\mathrm{d}s\,\right\vert \,%
\mathcal{Y}\right] \right] \\
& =\tilde{\mathbb{E}}\left[ \int_{0}^{t}i\varepsilon _{s}r_{s}^{j}\,\tilde{%
\mathbb{E}}[u_{s}\mid \mathcal{Y}]\,\mathrm{d}s\right] \\
& =\tilde{\mathbb{E}}\left[ \varepsilon _{t}\int_{0}^{t}\tilde{\mathbb{E}}%
[u_{s}\mid \mathcal{Y}]\,\mathrm{d}Y_{s}^{j}\right] ,
\end{align*}%
which completes the proof of (\ref{condU}).

\item[ii.] From Proposition \ref{prop:filterEqns:YisBM} we know that $%
\int_{0}^{\cdot }$ $\tilde{Z}_{s}\,\mathrm{d}\tilde{M}_{s}^{\varphi }$ is a
zero-mean martingale under $\mathbb{\tilde{P}}$. It is therefore integrable
and its conditional expectation is well defined. \ Notice that 
\begin{equation*}
\left\langle \tilde{M}^{\varphi },Y^{j}\right\rangle _{t}=\left\langle
M^{\varphi },W^{j}\right\rangle _{t}=\int_{0}^{t}\left( B^{j}\varphi +\frac{%
\partial \varphi }{\partial y_{j}}\right) \left( \bar{X}_{s}\right) \,%
\mathrm{d}s
\end{equation*}%
The rest of the proof of (\ref{condu''}) is similar to that of (\ref{condU}%
). Once again we choose $\varepsilon _{t}$ from the set $S_{t}$ and in this
case we obtain the following sequence of identities. 
\begin{align*}
\tilde{\mathbb{E}}\left[ \varepsilon _{t}\tilde{\mathbb{E}}\left[ \left.
\int_{0}^{t}\tilde{Z}_{s}\,\mathrm{d}\tilde{M}_{s}^{\varphi }\,\right\vert \,%
\mathcal{Y}\right] \right] ={}& \tilde{\mathbb{E}}\left[ \varepsilon
_{t}\int_{0}^{t}\tilde{Z}_{s}\,\mathrm{d}\tilde{M}_{s}^{\varphi }\right] \\
={}& \tilde{\mathbb{E}}\left[ \int_{0}^{t}\tilde{Z}_{s}\,\mathrm{d}\tilde{M}%
_{s}^{\varphi }\right] +\sum_{j=1}^{m}\tilde{\mathbb{E}}\left\langle
\int_{0}^{\cdot }i\varepsilon _{s}r_{s}^{j}\,\mathrm{d}Y_{s}^{j},\int_{0}^{%
\cdot }\tilde{Z}_{s}\,\mathrm{d}\tilde{M}_{s}^{\varphi }\right\rangle _{t} \\
={}& \tilde{\mathbb{E}}\left[ \int_{0}^{t}\tilde{Z}_{s}\,\mathrm{d}\tilde{M}%
_{s}^{\varphi }\right] +\sum_{j=1}^{m}\tilde{\mathbb{E}}\int_{0}^{t}i%
\varepsilon _{s}r_{s}^{j}\tilde{Z}_{s}\,\mathrm{d}\left\langle \tilde{M}%
^{\varphi },Y^{j}\right\rangle _{s} \\
={}& \sum_{j=1}^{m}\tilde{\mathbb{E}}\int_{0}^{t}i\varepsilon _{s}r_{s}^{j}%
\tilde{Z}_{s}\,\left( B^{j}\varphi +\frac{\partial \varphi }{\partial y_{j}}%
\right) \left( \bar{X}_{s}\right) \mathrm{d}s \\
={}& \sum_{j=1}^{m}\tilde{\mathbb{E}}\left[ \varepsilon _{t}\int_{0}^{t}%
\tilde{\mathbb{E}}\left[ \left. \left( B^{j}\varphi +\frac{\partial \varphi 
}{\partial y_{j}}\right) \left( \bar{X}_{s}\right) \tilde{Z}%
_{s}\,\right\vert \,\mathcal{Y}\right] \mathrm{d}Y_{s}^{j}\right] .
\end{align*}%
As the identities hold for an arbitrary choice of $\varepsilon _{t}\in
S_{t}, $ the proof of (\ref{condu''}) is complete
\end{description}
\end{proof}

\section{The Filtering Equations\label{section6}}

To simplify the analysis, we will impose onto $\tilde{Z}\,\ $a similar
condition to (\ref{condonh}). More precisely, we will assume that, 
\begin{equation}
\mathbb{\tilde{E}}\left[ \int_{0}^{t}\tilde{Z}_{s}\left\Vert h(\bar{X}%
_{s})\right\Vert ^{2}\,\mathrm{d}s\right] <\infty ,\quad \forall t>0.
\label{adcondonh}
\end{equation}%
Reverting back to $\mathbb{P}$, condition (\ref{adcondonh}) is equivalent to 
\begin{equation}
\mathbb{E}\left[ \int_{0}^{t}\left\Vert h(\bar{X}_{s})\right\Vert ^{2}\,%
\mathrm{d}s\right] <\infty ,\quad \forall t>0.  \label{adcondonh'}
\end{equation}%
From Corollary \ref{corgenuinemart}, it follows that $\tilde{Z}$ is an $%
H^{1}(\mathbb{\tilde{P})}$-martingale. Then $\left( \tilde{Z}_{\cdot
}-1\right) $ is a zero-mean martingale and $\mathbb{E}\left[ \left( \tilde{Z}%
_{\cdot }-1\right) _{t}^{\ast }\right] \,<1+$ $\mathbb{E}\left[ \tilde{Z}%
_{t}^{\ast }\right] \,<\infty $. Since \thinspace $\left\langle \tilde{Z}%
_{\cdot }-1\right\rangle _{t}=\int_{0}^{t}\tilde{Z}_{s}^{2}\left\vert h(\bar{%
X}_{s})\right\vert ^{2}ds$ the Burkholder-Davis-Gundy inequalities give%
\begin{equation}
\mathbb{E}\left[ \left( \int_{0}^{t}\tilde{Z}_{s}^{2}\left\vert h(\bar{X}%
_{s})\right\vert ^{2}ds\right) ^{1/2}\right] <\infty  \label{usedlater}
\end{equation}%
for all $t\geq 0$ and hence, for any $\varphi \in $\smallskip $b\mathcal{B}(%
\mathbb{S\times R}^{m}),$ the processes 
\begin{eqnarray*}
t &\rightarrow &\int_{0}^{t}\varphi (\bar{X}_{t})\tilde{Z}_{t}h(\bar{X}%
_{s})^{\top }\mathrm{d}Y_{s} \\
t &\rightarrow &\int_{0}^{t}\tilde{\mathbb{E}}[\varphi (\bar{X}_{t})\tilde{Z}%
_{t}h(\bar{X}_{s})^{\top }\mid \mathcal{Y}_{t}]\,\mathrm{d}Y_{s}
\end{eqnarray*}%
are zero-mean $H^{1}(\mathbb{\tilde{P})}$ martingales. In the following, for
any function $\varphi \in b\mathcal{B}(\mathbb{S\times R}^{m})$ such that $%
\varphi \in \mathcal{D}(A)$ and that has bounded partial derivatives in the $%
y$ direction we will denote by $D_{i}\varphi ,~~~j=1,\ldots ,m$ the
functions 
\begin{equation*}
D_{j}\varphi =h^{j}\left( \varphi +B^{j}\varphi +\frac{\partial \varphi }{%
\partial y_{j}}\right) ~~~~~~j=1,\ldots ,m.
\end{equation*}

\begin{theorem}
If conditions (\ref{condonh}) and (\ref{adcondonh}) are satisfied then, 
\begin{equation}
\tilde{\mathbb{E}}[\tilde{Z}_{t}\varphi (\bar{X}_{t})\mid \mathcal{Y}]={}\pi
_{0}(\varphi )+\int_{0}^{t}\tilde{\mathbb{E}}[\tilde{Z}_{s}A\varphi (\bar{X}%
_{s})\mid \mathcal{Y}]\,\mathrm{d}s+\sum_{j=1}^{m}\tilde{\mathbb{E}}[\tilde{Z%
}_{s}D_{j}\varphi (\bar{X}_{s})\mid \mathcal{Y}]\mathrm{d}Y_{s}^{j}
\label{intermediate1}
\end{equation}%
for any $\varphi \in b\mathcal{B}(\mathbb{S\times R}^{m})$ be a function
such that $\varphi ,\varphi ^{2}\in \mathcal{D}(A)$ and that has bounded
partial derivatives in the $y$ direction. In particular the process $\rho
_{t}$ satisfies the following evolution equation 
\begin{equation}
\rho _{t}(\varphi )=\rho _{0}(\varphi )+\int_{0}^{t}\rho _{s}\left( A\varphi
\right) \,\mathrm{d}s+\int_{0}^{t}\rho _{s}((h^{\top }+B^{\top })\varphi )\,%
\mathrm{d}Y_{s},\ \quad \tilde{\mathbb{P}}\text{-a.s.}\ \forall t\geq 0
\label{eq:filterEqns:zakaifunctint}
\end{equation}%
for any function $\varphi \in $\smallskip $b\mathcal{B}(\mathbb{S})$ be a
function such that $\varphi \in \mathcal{D}(A)$.
\end{theorem}

\begin{proof}
Using It\^{o}'s formula and integration-by-parts, we find 
\begin{eqnarray}
\mathrm{d}\left( \tilde{Z}_{t}\varphi (\bar{X}_{t})\right) &=&{}\tilde{Z}%
_{t}A\varphi (\bar{X}_{t})\,\mathrm{d}t+\tilde{Z}_{t}\mathrm{d}%
M_{t}^{\varphi }+\varphi (\bar{X}_{t})\tilde{Z}_{t}h^{\top }(\bar{X}_{t})\,%
\mathrm{d}Y_{t}+\sum_{j=1}^{m}\tilde{Z}_{t}h^{i}(\bar{X}_{t})\left\langle
M^{\varphi },Y^{i}\right\rangle _{t}  \notag \\
&=&\tilde{Z}_{t}\left[ A\varphi (\bar{X}_{t})+\sum_{j=1}^{m}h^{i}(\bar{X}%
_{t})\left( B^{i}\varphi \left( \bar{X}_{t}\right) +\frac{\partial \varphi }{%
\partial y_{i}}\left( \bar{X}_{t}\right) \right) \right] {}\mathrm{d}t+%
\tilde{Z}_{t}\mathrm{d}M_{t}^{\varphi }  \notag \\
&&+\varphi (\bar{X}_{t})\tilde{Z}_{t}h^{\top }(\bar{X}_{t})\,\mathrm{d}Y_{t}
\label{lll1} \\
&=&\tilde{Z}_{t}A\varphi (\bar{X}_{t})\mathrm{d}t+\tilde{Z}_{t}\mathrm{d}%
\tilde{M}_{t}^{\varphi }+\varphi (\bar{X}_{t})\tilde{Z}_{t}h^{\top }(\bar{X}%
_{t})\,\mathrm{d}Y_{t}.  \notag
\end{eqnarray}%
We next take the conditional expectation with respect to $\mathcal{Y}$ \ and
obtain 
\begin{eqnarray}
\tilde{\mathbb{E}}[\tilde{Z}_{t}\varphi (\bar{X}_{t}) &\mid &\mathcal{Y}]=%
\tilde{\mathbb{E}}[\tilde{Z}_{0}\varphi (\bar{X}_{t})\mid \mathcal{Y}%
]+\int_{0}^{t}\tilde{\mathbb{E}}[\tilde{Z}_{t}A\varphi (\bar{X}_{t})\mid 
\mathcal{Y}]\,\mathrm{d}s  \notag \\
&&+\tilde{\mathbb{E}}\left[ \int_{0}^{t}\tilde{Z}_{s}\mathrm{d}\tilde{M}%
_{s}^{\varphi }\mid \mathcal{Y}\right] +\tilde{\mathbb{E}}\left[
\int_{0}^{t}\varphi (\bar{X}_{s})\tilde{Z}_{s}h^{\top }(\bar{X}_{s})\,%
\mathrm{d}Y_{s}\mid \mathcal{Y}\right] ,  \label{ll2}
\end{eqnarray}%
where we have used Fubini's theorem (the conditional version) to get the
second term on the right hand side of (\ref{ll2}). Observe that, since $%
\tilde{Z}$ is an $H^{1}(\mathbb{\tilde{P})}$-martingale, we have 
\begin{equation*}
\tilde{\mathbb{E}}\left[ \left( \int_{0}^{t}\tilde{Z}_{s}^{2}\,\mathrm{d}%
s\right) ^{1/2}\right] \leq \sqrt{t}\tilde{\mathbb{E}}\left[ \tilde{Z}%
_{s}^{\ast }\,\right] <\infty .
\end{equation*}%
Also from (\ref{usedlater}) we get that%
\begin{equation*}
\tilde{\mathbb{E}}\left[ \left( \int_{0}^{t}\left( \varphi (\bar{X}_{s})%
\tilde{Z}_{s}h^{j}(\bar{X}_{s})\right) ^{2}\,\mathrm{d}s\right) ^{1/2}\right]
\leq \left\vert \left\vert \varphi \right\vert \right\vert \mathbb{E}\left[
\left( \int_{0}^{t}\tilde{Z}_{s}^{2}\left\vert h(\bar{X}_{s})\right\vert
^{2}ds\right) ^{1/2}\right] <\infty .
\end{equation*}%
In other words condition (\ref{scondu}) is satisfied for $u=\varphi \tilde{Z}%
h^{j}$. The identity (\ref{intermediate1}) then follows from (\ref{ll2}) by
applying (\ref{condU}) and (\ref{condu''}). Identity (\ref%
{eq:filterEqns:zakaifunctint}) follows immediately after observing that the
terms containing the partial derivatives in the $y$ direction $\frac{%
\partial \varphi }{\partial y_{i}}$ are zero since the function no longer
depends on $y$.
\end{proof}

\begin{theorem}
If conditions (\ref{condonh}) and (\ref{adcondonh}) are satisfied then the
conditional distribution of the signal $\pi _{t}$ satisfies the following
evolution equation 
\begin{align}
\pi _{t}(\varphi )={}& \pi _{0}(\varphi )+\int_{0}^{t}\pi _{s}(A\varphi )\,%
\mathrm{d}s  \notag \\
& +\int_{0}^{t}\left( \pi _{s}(\varphi h^{\top })-\pi _{s}(h^{\top })\pi
_{s}(\varphi )+\pi _{t}(B^{\top }\varphi )\right) (\mathrm{d}Y_{s}-\pi
_{s}(h)\,\mathrm{d}s),  \label{eq:filterEqns:ks}
\end{align}%
for any $\varphi \in \mathcal{D}(A)$.
\end{theorem}

\begin{proof}
Since $A\mathbf{1=}0$, it follows from (\ref{generatorforXand Y}) that $M^{%
\mathbf{1}}\equiv 0,$ which together with (\ref{covariance}) implies that 
\begin{equation*}
\int_{0}^{t}B^{i}\mathbf{1}\left( \bar{X}_{s}\right) ds=0,
\end{equation*}%
for any $t\geq 0$ and $i=1,...,m$, so 
\begin{equation*}
\sum_{j=1}^{m}\int_{0}^{t}\rho _{s}\left( h^{j}B^{j}\mathbf{1}\right) \,%
\mathrm{d}s=0.
\end{equation*}%
\ Hence, from (\ref{eq:filterEqns:zakaifunctint}), one obtains that $\rho
_{t}(\mathbf{1})$ satisfies the following equation 
\begin{equation*}
\rho _{t}(\mathbf{1})=1+\int_{0}^{t}\rho _{s}(h^{\top })\,\mathrm{d}Y_{s}.
\end{equation*}%
Let $\left( U_{n}\right) _{n>0}$ be the sequence of stopping times%
\begin{equation*}
U_{n}=\inf \left\{ t\geq 0\left\vert \rho _{t}(\mathbf{1})\leq \frac{1}{n}%
\right. \right\} .
\end{equation*}%
Then 
\begin{equation*}
\rho _{t}^{U_{n}}(\mathbf{1})=\rho _{t\wedge U_{n}}(\mathbf{1}%
)=1+\int_{0}^{t\wedge U_{n}}\rho _{s}(h^{\top })\,\mathrm{d}Y_{s},
\end{equation*}%
We apply It\^{o}'s formula to the stopped process $t\rightarrow \rho
_{t\wedge U_{n}}(\mathbf{1})$ and the function $x\mapsto \frac{1}{x}$ to
obtain that 
\begin{equation}
\frac{1}{\rho _{t}^{U_{n}}(\mathbf{1})}=1-\int_{0}^{t\wedge U_{n}}\frac{\rho
_{s}(h^{\top })\,}{\rho _{s}(\mathbf{1})^{2}}\mathrm{d}Y_{s}+\int_{0}^{t%
\wedge U_{n}}\frac{\rho _{s}(h^{\top })\rho _{s}(h)}{\rho _{s}(\mathbf{1}%
)\,^{3}}\,\mathrm{d}s  \label{eq:filterEqns:rhotof1}
\end{equation}%
By using (stochastic) integration by parts, (\ref{eq:filterEqns:rhotof1}),
the equation for $\rho _{t}(\varphi )$ and the Kallianpur--Striebel formula,
we obtain 
\begin{eqnarray*}
\frac{\rho _{t}^{U_{n}}(\varphi )}{\rho _{t}^{U_{n}}(\mathbf{1})} &=&\pi
_{0}\left( \varphi \right) +\int_{0}^{t\wedge U_{n}}\pi _{s}\left( A\varphi
\right) \,\mathrm{d}s+\int_{0}^{t\wedge U_{n}}\pi _{s}((h^{\top }+B^{\top
})\varphi )\,\mathrm{d}Y_{s}-\int_{0}^{t\wedge U_{n}}\pi _{s}(\varphi )\pi
_{s}(h^{\top })\mathrm{d}Y_{s} \\
&&+\int_{0}^{t\wedge U_{n}}\pi _{s}(\varphi )\pi _{s}(h^{\top })\pi _{s}(h)\,%
\mathrm{d}s-\int_{0}^{t\wedge U_{n}}\pi _{s}((h^{\top }+B^{\top })\varphi
)\,\pi _{s}(h)\mathrm{d}s
\end{eqnarray*}%
As $\lim_{n\rightarrow \infty }U_{n}=\infty $ almost surely, we obtain the
result by taking the limit as $n$ tends to infinity.
\end{proof}

\begin{remark}
The jump-diffusion example and the change detection model discussed in
Section \ref{examples} both satisfy conditions (\ref{condonh}) and (\ref%
{adcondonh'}). Therefore the two previous theorems can be applied to these
two cases.
\end{remark}

\textbf{Acknowledgments.} The authors are grateful to J. Ruf for setting us
straight on the provenance of Corollary \ref{corgenuinemart} and on other
points in the paper.

\end{document}